\documentclass[11pt]{amsart}
\usepackage{amsmath,amssymb,amsthm}
\usepackage[margin=1in]{geometry}

\newtheorem{theorem}{Theorem}[section]
\newtheorem{proposition}[theorem]{Proposition}
\newtheorem{lemma}[theorem]{Lemma}
\newtheorem{corollary}[theorem]{Corollary}
\theoremstyle{definition}
\newtheorem{definition}[theorem]{Definition}
\theoremstyle{remark}
\newtheorem{remark}[theorem]{Remark}
\newtheorem{example}[theorem]{Example}
\theoremstyle{plain}
\newtheorem{introtheorem}{Theorem}

\newtheorem{introtheoremprime}{Theorem}

\newtheorem{introtheoremsecond}{Theorem}

\newtheorem*{introproposition}{Proposition}
\newtheorem*{introtheoremnamed}{Theorem}

\DeclareMathOperator{\sign}{sign}
\DeclareMathOperator{\cone}{cone}
\DeclareMathOperator{\conv}{conv}
\newcommand{\R}{\mathbb{R}}
\newcommand{\A}{\mathcal A}

\title{Topological degree of fewnomial maps}
\author{Fr\'ed\'eric Bihan}
\address{Laboratoire de Math\'ematiques, Universit\'e Savoie Mont Blanc, 73376 Le Bourget-du-Lac Cedex, France}
\email{frederic.bihan@univ-smb.fr}
\date{\today}

\begin{document}

\begin{abstract}
We study the topological degree on the positive orthant of a map $F\colon\R^n_{>0}\to\R^n$ whose components are real polynomials, possibly with real exponents, with the same support $\A$ of $N=n+k+1$ monomials. This degree is well defined under a half-space condition on the columns of the coefficient matrix indexed by the faces of the Newton polytope, which holds for generic exponents and coefficients. We give three formulas for it, all in terms of signs of maximal minors. The first one is a Viro formula: the degree is a signed count of the simplices of a regular triangulation of $\A$ which are positively decorated in the sense of Bihan, Santos and Spaenlehauer. The second one is a sum over the facets of the Newton polytope, and the third one a sum over the vertices of the polyhedron of the Gale dual system. As a consequence, the degree only depends on the oriented matroids of the exponent and coefficient matrices restricted to the vertices of the Newton polytope. We also obtain a Descartes rule of signs for the degree and, in the generic case, the upper bound $|\deg F|<\frac12\min(M(n,N),M(k,N))$, where $M(d,N)$ is the largest number of facets of a $d$-polytope with $N$ vertices. This bound is sharp when $n=2$ or $k=2$. Using the H-duality of Bihan, Dickenstein and Forsg\aa rd, we show that the set of possible degrees is symmetric in $n$ and $k$. Since a nonzero degree forces a positive zero, our formulas give sign conditions for the existence of positive solutions, as well as parity statements for their number.
\end{abstract}

\maketitle

\section*{Introduction}

Let $\A=\{a_1,\dots,a_N\}\subset\R^n$ be a finite set which affinely spans $\R^n$, write $N=n+k+1$, and let $C=(c_{li})$ be a real $n\times N$ matrix with columns $C_1,\dots,C_N$. We consider the \emph{fewnomial map}
\[
F=(f_1,\dots,f_n)\colon\R^n_{>0}\to\R^n,\qquad f_l(x)=\sum_{i=1}^Nc_{li}\,x^{a_i},\qquad\text{that is,}\qquad F(x)=\sum_{i=1}^NC_i\,x^{a_i}.
\]
All its components have the same support $\A$, and we allow real exponents. In this paper we study the \emph{topological degree} $\deg F$ of $F$ on the positive orthant. When the positive zeros of $F$ are finitely many and nondegenerate, it is given by
\[
\deg F=\sum_{x\in\R^n_{>0},\ F(x)=0}\sign\det DF(x).
\]
In general $\deg F$ is the Brouwer degree, which is defined as soon as the positive zeros of $F$ form a compact set (Definition~\ref{def:deg}).

Our starting point is the moment map, which identifies $\R^n_{>0}$ with the interior of the Newton polytope $P=\conv(\A)$. Through this identification, $F$ extends to a map $\hat F$ on $P$. If no truncation of $F$ to a proper face of $P$ has a positive zero, then $\deg F$ is defined, and it is the degree of $\hat F/|\hat F|$ on $\partial P$ (Proposition~\ref{prop:degree}). In that case $F$ has a positive zero as soon as $\deg F\neq0$. Moreover, the number of nondegenerate positive zeros is at least $|\deg F|$, and it has the parity of $\deg F$. This parity can be read on each of our formulas, as the parity of their number of terms. It gives existence criteria (Corollary~\ref{cor:odd}), and it also refines upper bounds. For instance, the system of El Hilany \cite{EH}, with two equations, five monomials and seven positive solutions, has degree $\pm1$, so the bound $8$ of \cite{PartI} becomes $7$, which is attained (Example~\ref{ex:EH}).

Let $A$ be the $(n+1)\times N$ matrix with columns $(1,a_i)$. We say that we are in the \emph{generic case} when the following two conditions hold:
\begin{itemize}
\item[(U$_A$)] all the maximal minors of $A$ are nonzero, that is, no $n+1$ points of $\A$ lie on an affine hyperplane;
\item[(U$_C$)] all the maximal minors of $C$ are nonzero, that is, any $n$ columns of $C$ are linearly independent.
\end{itemize}

\subsection*{Two formulas in the generic case}
Unless otherwise specified, the columns of a square submatrix of $A$ or $C$ are taken in increasing order of their indices. Under (U$_A$), every facet $\Gamma$ of $P$ contains exactly $n$ points of $\A$, and we identify $\Gamma$ with the set of their indices.

\begin{introtheorem}[Newton polytope formula]\label{thm:A}
Assume (U$_A$) and (U$_C$), and let $u\in\R^n$ be a vector which does not lie in any cone spanned by $n-1$ columns of $C$. For each facet $\Gamma$ of $P$ choose an index $i_\Gamma\notin\Gamma$. Then
\[
\deg F=(-1)^n\sum_{\substack{\Gamma\ \text{facet of}\ P\\ u\in\cone(C_s,\ s\in\Gamma)}}\sign\det C^\Gamma\cdot\sign\det A^{\Gamma\cup\{i_\Gamma\}},
\]
where $C^\Gamma$ is the submatrix of $C$ with the columns indexed by $\Gamma$, and $A^{\Gamma\cup\{i_\Gamma\}}$ the submatrix of $A$ with the columns indexed by $\Gamma$, followed by the column $i_\Gamma$. The sign of $\det A^{\Gamma\cup\{i_\Gamma\}}$ does not depend on $i_\Gamma$, since the points of $\A$ off $\Gamma$ lie on the same side of it.
\end{introtheorem}

We write $\varepsilon_\Gamma=\sign\det A^{\Gamma\cup\{i_\Gamma\}}$. For $n=1$, Theorem~\ref{thm:A} is the parity part of Descartes' rule of signs. For $n=2$, it says that $\deg F$ is the winding number of the polygon formed by the columns of $C$ indexed by the vertices of the Newton polygon, and this remains true under the condition (H) below (Corollary~\ref{cor:n2}).

The second formula uses the Gale duality of Bihan and Sottile \cite{BSGale,BS}. Let $B$ be an $N\times k$ matrix whose columns form a basis of $\ker A$, and let $D$ be an $N\times(k+1)$ matrix whose columns form a basis of $\ker C$, with last row $(0,\dots,0,1)$. We denote by $D_i$ the $i$-th row of $D$, put $p_i(y)=\langle D_i,(y,1)\rangle$, and consider the polyhedron
\[
\Delta=\{y\in\R^k:\ p_i(y)>0,\ i=1,\dots,N\}.
\]
Recall that the positive zeros of $F$ correspond to the zeros in $\Delta$ of the Gale system $\psi_j=\sum_iB_{ij}\log p_i$, $j=1,\dots,k$. Let $V(\Delta)$ be the set of the sets $S\subset[N]$ of size $k$ such that the hyperplanes $\{p_s=0\}$, $s\in S$, meet at a vertex of $\overline\Delta$, possibly at infinity.

\begin{introtheorem}[Gale dual formula]\label{thm:B}
Assume (U$_A$) and (U$_C$), and let $e\in\R^k$ be a vector which does not lie in any cone spanned by $k-1$ of the rows $-B_1,\dots,-B_N$. If $\Delta=\emptyset$, then $\deg F=0$. Otherwise, for each $S\in V(\Delta)$ choose an index $i_S\notin S$. Then
\[
\deg F=\pm\sum_{\substack{S\in V(\Delta)\\ e\in\cone(-B_s,\ s\in S)}}\sign\det B_S\cdot\sign\det(D_S;D_{i_S}),
\]
where $B_S$ is the submatrix of $B$ with the rows indexed by $S$, and $(D_S;D_{i_S})$ the submatrix of $D$ with the rows indexed by $S$, followed by the row $i_S$. The sign of $\det(D_S;D_{i_S})$ does not depend on $i_S$, and the global sign only depends on the choices of $B$ and $D$.
\end{introtheorem}

We write $\omega(S)=\sign\det B_S\cdot\sign\det(D_S;D_{i_S})$ and $\deg(\Delta)=\sum_{S\in V(\Delta),\ e\in\cone(-B_s,\ s\in S)}\omega(S)$. For $k=1$, Theorem~\ref{thm:B} gives back the parity part of the Descartes rule for circuits \cite{BD}. For $k=2$, it says that $\deg F$ is a winding number of the rows of $B$ (Proposition~\ref{prop:k2}).

\subsection*{The general case}
Both formulas extend beyond the generic case. The condition we need is the following \emph{half-space condition}:
\begin{itemize}
\item[(H)] for every proper face $G$ of $P$, the columns $C_i$ with $a_i\in G$ lie in an open half-space of $\R^n$.
\end{itemize}
This condition holds under (U$_A$) and (U$_C$), but it is much weaker. Indeed, it puts no condition on the monomials whose exponents lie in the interior of $P$, and it allows any number of points of $\A$ on each face of $P$, as for the set of all lattice points of a lattice polytope. It guarantees that $\deg F$ is defined, and that it does not change under positive rescalings of the columns of $C$.

Observe that (H) is a condition on the oriented matroids of $A$ and $C$, that is, on the signs of their maximal minors. Indeed, the faces of $P$, seen as subsets of $[N]$, are determined by the oriented matroid of $A$: a set $G$ is a face if and only if there is a covector, namely the sign vector of $(\ell(a_i))_i$ for an affine function $\ell$, which vanishes on $G$ and is positive outside $G$ \cite{BLSWZ}. On the other hand, by Gordan's theorem of the alternative \cite[Section~7.8]{Schrijver}, finitely many vectors of $\R^n$ lie in an open half-space if and only if they have no linear dependence with nonnegative coefficients, not all zero, that is, if and only if $0$ is not in their convex hull. In other words, (H) says that for every proper face $G$ of $P$, no vector of the oriented matroid of the columns $C_i$, $i\in G$, has all its signs nonnegative.

\begin{introtheoremprime}[Newton polytope formula, general case]\label{thm:Ap}
Assume (H), and let $T$ be a triangulation of $\partial P$ with vertices in $\A$ such that every face of $P$ is a union of simplices of $T$. Let $u$ be as in Theorem~\ref{thm:A}, and for each $(n-1)$-simplex $\tau$ of $T$ let $i_\tau$ be an index such that $a_{i_\tau}$ is off the facet containing $\tau$. Then
\[
\deg F=(-1)^n\sum_{\substack{\tau\in T,\ \dim\tau=n-1\\ u\in\cone(C_s,\ s\in\tau)}}\sign\det C^\tau\cdot\sign\det A^{\tau\cup\{i_\tau\}} .
\]
\end{introtheoremprime}

When $P$ is simplicial, we can apply Theorem~\ref{thm:Ap} to the boundary complex of $P$. This shows that Theorem~\ref{thm:A} holds under (H) alone, with the facets replaced by their sets of vertices. In fact, only the vertices of $P$ matter. Let $V\subset[N]$ be the set of indices $i$ such that $a_i$ is a vertex of $P$, and let $F_V(x)=\sum_{i\in V}C_ix^{a_i}$ be the truncation of $F$ to these vertices.

\begin{introproposition}[Vertex reduction]
Assume (H). Then $F_V$ satisfies (H), and $\deg F=\deg F_V$.
\end{introproposition}

Thus, under (H), the monomials of $F$ whose exponents are not vertices of the Newton polytope, whether they lie on its boundary or in its interior, do not change the degree (Proposition~\ref{lem:vertex}). This is not true without (H). For instance, for the maps $F_\lambda$ of Remark~\ref{rem:general} with $\lambda>1$, the truncation to the vertices of the Newton polygon is $(1+x^2,\,x^2-y)$, which has no positive zero, while $\deg F_\lambda=-1$.

The vector $u$ can be chosen in a combinatorial way, as a lexicographic perturbation of a basis of columns of $C$. Assume that $C$ has rank $n$, and let $C_{j_1},\dots,C_{j_n}$ be the first basis of $\R^n$ formed by columns of $C$. In other words, $j_1=1$, and $j_r$ is the smallest index such that $C_{j_r}$ does not lie in the span of $C_{j_1},\dots,C_{j_{r-1}}$. For a set $\tau\subset[N]$ of size $n$ such that the columns $C_s$, $s\in\tau$, form a basis, we write $C_{j_r}=\sum_{s\in\tau}\mu^\tau_{s,r}C_s$. By Cramer's rule, $\mu^\tau_{s,r}=\det C^\tau_{s\leftarrow r}/\det C^\tau$, where $C^\tau_{s\leftarrow r}$ is obtained from $C^\tau$ by replacing $C_s$ by $C_{j_r}$. We say that $\tau$ is \emph{selected} if, for every $s\in\tau$, the first nonzero coordinate of $(\mu^\tau_{s,1},\dots,\mu^\tau_{s,n})$ is positive. Note that this only depends on the signs of the maximal minors of $C$.

\begin{introtheoremsecond}[Combinatorial formula]\label{thm:comb}
Assume (H) and that $C$ has rank $n$, and let $T$ be as in Theorem~\ref{thm:Ap}. Then
\[
\deg F=(-1)^n\sum_{\substack{\tau\in T,\ \dim\tau=n-1\\ \tau\ \text{selected}}}\sign\det C^{\tau}\cdot\sign\det A^{\tau\cup\{i_\tau\}}.
\]
If $C$ has rank less than $n$, then $\deg F=0$. Consequently, under (H), $\deg F$ only depends on the oriented matroids of $A$ and $C$.
\end{introtheoremsecond}

A canonical choice of $T$ is the pulling triangulation \cite[Section~4.3]{DLRS}, and Remark~\ref{rem:comb} gives the analogous version of Theorem~\ref{thm:Bp}. The last assertion of Theorem~\ref{thm:comb} holds because a triangulation $T$ as in Theorem~\ref{thm:Ap} only depends on the oriented matroid of $\A$. Together with the vertex reduction, we obtain the main structural result of the paper.

\begin{introtheoremnamed}[Oriented matroid invariance]
Assume (H). Then $\deg F$ only depends on the oriented matroids of the columns of $A$ and of $C$ indexed by the vertices of $P$.
\end{introtheoremnamed}

In other words, under (H) the degree does not see the monomials of $F$ whose exponents are not vertices of $P$. It does not see the actual values of the exponents and coefficients of the other monomials either, but only the signs of the maximal minors of the corresponding columns of $A$ and $C$ (Theorem~\ref{cor:om}). Since (H) itself only depends on the oriented matroids of $A$ and $C$, the degree is a function on the set of pairs of oriented matroids satisfying (H). Again, this is not true without (H): the degree can then change under a positive rescaling of the columns of $C$, which preserves the oriented matroids (Remark~\ref{rem:general}).

\begin{introtheoremprime}[Gale dual formula, general case]\label{thm:Bp}
Assume (U$_C$) and (H). Then the formula of Theorem~\ref{thm:B} holds.
\end{introtheoremprime}

On the Gale side, the faces of $P$ correspond to the \emph{cofaces}, that is, to the sets $S$ such that the rows $B_s$, $s\in S$, have a positive linear dependence. The condition (H) says that for every coface $S$, the hyperplanes $\{p_s=0\}$, $s\in S$, have no common point on $\overline\Delta$ (Lemma~\ref{lem:gale-N}). The vertex reduction also has a Gale dual version: only the vertices of $\overline\Delta$ on the face cut out by the hyperplanes $\{p_i=0\}$, $i\notin V$, contribute to the formula of Theorem~\ref{thm:B} (Remark~\ref{rem:gale-vertex}).

\subsection*{Half signed counts}
The auxiliary vectors $u$ and $e$ can also be replaced by a generic change of basis. Let $C'$ be the matrix formed by the first $n-1$ rows of $C$, and assume that neither $e_n$ nor $-e_n$ lies in a cone spanned by $n-1$ columns of $C$; this holds in particular if all the maximal minors of $C'$ are nonzero, and it can always be achieved by replacing $F$ by $MF$ for a generic $M\in GL_n(\R)$, which does not change the positive zeros. Then both $u=e_n$ and $u=-e_n$ can be used in Theorem~\ref{thm:Ap}, and adding the two formulas we get
\[
\deg F=\frac{(-1)^n}2\sum_{\substack{\tau\in T,\ \dim\tau=n-1\\ (C'_s)_{s\in\tau}\ \text{positively span}\ \R^{n-1}}}\sign\det C^\tau\cdot\sign\det A^{\tau\cup\{i_\tau\}}
\]
(Proposition~\ref{prop:ends-newton}). The terms of this sum correspond to the ends of the curve defined by the first $n-1$ components of $F$: a compact component of this curve contributes nothing to $\deg F$, and a noncompact one contributes $\pm1$ or $0$, according to whether $f_n$ takes opposite signs or the same sign near its two ends (Remark~\ref{rem:curve-newton}). In the same way, after a generic change of basis of $\ker A$, $\deg F$ is $\pm\frac12$ times the sum of the signs $\omega(S)$ over the vertices $S\in V(\Delta)$ such that the first $k-1$ coordinates of the rows $B_s$, $s\in S$, positively span $\R^{k-1}$ (Proposition~\ref{prop:ends-gale}). For generic data, these vertices are the limits of the ends of the curve $\mathcal C\subset\Delta$ defined by the first $k-1$ equations of the Gale system, and each of them is reached by exactly one end \cite{PartI}. Thus $\deg F$ is half the signed count of the ends of $\mathcal C$. This is actually where this paper started. In \cite{PartI}, upper bounds for the number of positive solutions are obtained by a Khovanskii--Rolle argument on the Gale side, and the number of noncompact components of $\mathcal C$, which bounds $|\deg F|$, is bounded by half the number of these vertices.

\subsection*{The method}
We give several formulas for the degree, and several proofs of them, since each point of view explains a different feature. We prove Theorems~\ref{thm:A} and~\ref{thm:Ap} directly on the boundary of $P$. Under (H), the restriction of $\hat F$ to $\partial P$ can be deformed, without zeros, into the piecewise linear map which is affine on the simplices of $T$ and sends $a_i$ to $C_i$. In the generic case, it is already this map. We then compute the degree of this piecewise linear map by counting the preimages of the direction $u$.

The same piecewise linear argument applies to the map $\Pi_h\colon P\to\R^n$ which is affine on the simplices of the regular triangulation $T_h$ of $\A$ induced by heights $h$, and which sends $a_i$ to $C_i$. It shows that $\deg F$ is the degree of $\Pi_h$ on $P$. Recall from Bihan, Santos and Spaenlehauer \cite{BSS} that a simplex $\sigma$ of $T_h$ is \emph{positively decorated} by $C$ if the columns $C_l$, $l\in\sigma$, positively span $\R^n$. Under (U$_C$), $\Pi_h$ vanishes exactly once in each positively decorated simplex, and nowhere else. The local degree at this zero is a sign $s(\sigma)$, namely the product of the orientations of $\sigma$ induced by the points $a_l$ and by the columns $C_l$, $l\in\sigma$. This sign was introduced in \cite[Section~5]{BSS}, where it is shown that it alternates between adjacent positively decorated simplices \cite[Proposition~5.3]{BSS}. We obtain the \emph{Viro formula}
\[
\deg F=\sum_{\substack{\sigma\in T_h\\ \sigma\ \text{positively decorated}}}s(\sigma)
\]
(Proposition~\ref{prop:viro-pl}). Writing it with Gale dual matrices, through the relation between complementary maximal minors of a matrix and of its Gale dual, we get Theorems~\ref{thm:B} and~\ref{thm:Bp}.

The name comes from Viro's construction. Multiplying the coefficient of $x^{a_i}$ by $t^{h_i}$ preserves (H), hence the degree. For $t$ small, the positive zeros of the deformed map are the zeros constructed in \cite[Theorem~3.4]{BSS} near the positively decorated simplices, with Jacobian signs $s(\sigma)$, and there are no others (Theorem~\ref{thm:viro}). On the Gale side, the Viro deformation becomes the equation $\psi=\nu e$, where $\nu=-\log t\to+\infty$ and $e=-B^Th$. Its solutions concentrate near the vertices $S$ of $\overline\Delta$ such that $e\in\cone(-B_s,\ s\in S)$ (Remark~\ref{rem:gale-system}), which explains the role of $e$ in Theorem~\ref{thm:B}. Conversely, Theorem~\ref{thm:A} is itself a Viro formula, for the coning triangulation from an auxiliary interior point decorated by $-u$ (Section~\ref{sec:coning}).

\subsection*{Consequences}
The half signed counts give a Descartes rule of signs for the degree. We say that $m$ vectors of $\R^{m-1}$ form a \emph{sign change} if they positively span $\R^{m-1}$; for $m=2$, this means that two real numbers are nonzero and have opposite signs.

\begin{introtheorem}[Descartes rule of signs for the degree]\label{thm:descartes-intro}
Assume (H).
\begin{enumerate}
\item Let $T$ be a triangulation of $\partial P$ as in Theorem~\ref{thm:Ap}, and let $g_1,\dots,g_{n-1}$ be linear combinations of $f_1,\dots,f_n$ such that any $n-1$ of their coefficient vectors at the vertices of $T$ are linearly independent. Then $2|\deg F|$ is at most the number of $(n-1)$-simplices $\tau$ of $T$ at which the coefficient vectors of $(g_1,\dots,g_{n-1})$ form a sign change, that is, such that the truncated system $g_1|_\tau=\dots=g_{n-1}|_\tau=0$ has a positive solution.
\item Assume moreover (U$_C$), and let $b_1,\dots,b_{k-1}$ be affine relations of $\A$ such that any $k-1$ of the vectors $(b_1(i),\dots,b_{k-1}(i))$, $i\in[N]$, are linearly independent. Then $2|\deg F|$ is at most the number of vertices $S\in V(\Delta)$ such that the vectors $(b_1(s),\dots,b_{k-1}(s))$, $s\in S$, form a sign change.
\end{enumerate}
\end{introtheorem}

This is Theorem~\ref{thm:descartes}, where the linear independence conditions are replaced by weaker admissibility conditions, which hold for generic choices. Under (U$_A$), we can take for $T$ the boundary complex of $P$, so that the simplices $\tau$ are the facets of $P$. By the vertex reduction, we can also take for $T$ a triangulation of $\partial P$ whose vertices are the vertices of $P$. For $n=2$, we take for $T$ the edges of the Newton polygon. An edge is a sign change of a linear combination $g$ of $f_1$ and $f_2$ if and only if the coefficients of $g$ at its two endpoints have opposite signs. Thus $|\deg F|$ is at most half the number of cyclic sign changes, at the vertices of the Newton polygon, of the coefficients of any linear combination of $f_1$ and $f_2$ whose coefficients at these vertices are nonzero. In particular, if some linear combination has coefficients of the same sign at all the vertices, then $\deg F=0$. Similarly, for $k=2$, $|\deg F|$ is at most half the number of cyclic sign changes along the edges of $\overline\Delta$ of any affine relation of $\A$ whose values are nonzero, and for $n=1$, Theorem~\ref{thm:A} is the parity part of the classical Descartes rule. In this sense, Theorem~\ref{thm:descartes-intro} is a Descartes rule of signs in any dimension.

Bounding the number of terms of the half signed counts by the Upper Bound Theorem \cite{McM}, we obtain the following result.

\begin{introtheorem}\label{thm:C}
Assume (U$_A$) and (U$_C$). If $n,k\ge2$, then
\[
|\deg F|<\tfrac12\min\big(M(n,N),M(k,N)\big),\qquad M(d,N)=\binom{N-\lceil d/2\rceil}{\lfloor d/2\rfloor}+\binom{N-\lfloor d/2\rfloor-1}{\lceil d/2\rceil-1},
\]
where $M(d,N)$ is the number of facets of a cyclic $d$-polytope with $N$ vertices. If $n=1$ or $k=1$, then $|\deg F|\le1$.
\end{introtheorem}

This bound grows like $N^{\min(\lfloor n/2\rfloor,\lfloor k/2\rfloor)}$, which is much slower than the known upper bounds for the number of positive zeros \cite{BS,PartI}. Its symmetry in $n$ and $k$ is explained by the H-duality of Bihan, Dickenstein and Forsg\aa rd \cite{BDF24}. This duality associates with $F$ a fewnomial map $H$ with $k$ components in $k$ variables and the same number $N$ of monomials. In \cite{BDF24} it was asked whether the H-duality preserves the local multiplicity at corresponding zeros. In general $F$ and $H$ do not have the same number of positive zeros. However, our formulas show that the H-duality preserves the topological degree. Indeed, it exchanges the formula of Theorem~\ref{thm:A} for $F$ and the formula of Theorem~\ref{thm:B} for $H$, so that $\deg H=\pm\deg F$ (Theorem~\ref{thm:h-duality}). As a surprising consequence, the set of the absolute values of the degrees is symmetric in $n$ and $k$.

\begin{introtheorem}\label{thm:D}
Let $\mathcal D(n,k)$ be the set of the degrees $|\deg F|$ of the fewnomial maps with $n$ components and $N=n+k+1$ monomials satisfying (U$_A$) and (U$_C$). Then $\mathcal D(n,k)=\mathcal D(k,n)$.
\end{introtheorem}

\begin{introtheorem}\label{thm:E}
For all $n,k\ge1$,
\[
\mathcal D(n,1)=\mathcal D(1,n)=\{0,1\},\qquad \mathcal D(n,2)=\mathcal D(2,n)=\Big\{0,1,\dots,\Big\lfloor\frac n2\Big\rfloor+1\Big\}.
\]
\end{introtheorem}

Thus, for $n=2$ or $k=2$, the bound $\lfloor(N-1)/2\rfloor$ of Theorem~\ref{thm:C} is sharp. Note that the number of positive zeros can be larger: for $n=k=2$ the degree is at most $2$, while El Hilany's system has $7$ positive solutions. For $n,k\ge3$ we do not know the largest degree. The first open case is $n=k=3$, where Theorem~\ref{thm:C} gives $|\deg F|\le4$.

\subsection*{Existence of positive solutions}
Since $\deg F\neq0$ forces a positive zero, our formulas give sign conditions for the existence of positive solutions. These conditions only involve signs of maximal minors, so they hold for all the coefficient matrices with the same oriented matroid.

\begin{introtheorem}\label{thm:F}
Assume (H). Then $\deg F\neq0$, so that $F$ has at least one positive zero, as soon as one of the following sums has an odd number of terms, or has at least one term and all its terms equal:
\begin{enumerate}
\item the sum of Theorem~\ref{thm:Ap}, for some triangulation $T$ and some vector $u$ as there;
\item under (U$_C$), the Viro sum $\sum s(\sigma)$, with the signs $s(\sigma)$ of \cite[Section~5]{BSS}, over the simplices of $T_h$ positively decorated by $C$, for some generic $h\in\R^N$;
\item under (U$_C$) and if $\Delta\neq\emptyset$, the sum of Theorem~\ref{thm:B}, for some vector $e$ as there.
\end{enumerate}
If the terms of the sum are all equal and the positive zeros are nondegenerate, their number is at least the number of terms.
\end{introtheorem}

The odd case is a parity criterion (Corollary~\ref{cor:odd}): if a sum has an odd number of terms for one choice, then it has an odd number of terms for every choice. The equal-sign case also applies to maps of even degree. For instance, consider
\[
F(x,y)=\big(3x-3x^3+x^4y^2+x^2y^3-3y^2,\ 2x^3-3x^4y^2+3x^2y^3-2y^2\big).
\]
Its Newton polygon is a pentagon, and its columns $(3,0)$, $(-3,2)$, $(1,-3)$, $(1,3)$, $(-3,-2)$ at the consecutive vertices wind twice around the origin. Every vector $u$ selects two edges with the same sign, so $\deg F=2$ and $F$ has at least two positive zeros (Example~\ref{ex:deg2}). We use this map as a running example, on which we compute all our formulas (Examples~\ref{ex:run}, \ref{ex:run-A}, \ref{ex:deg2}, \ref{ex:run-gale} and the examples following Theorems~\ref{thm:descartes} and~\ref{thm:h-duality}). The criteria of Bihan, Dickenstein and Giaroli \cite{BDG} rely on a different principle. They show that the Gale dual map points inwards along the boundary of the polyhedron $\Delta$, so that its degree is $\pm1$. In particular, they cannot apply to this example.

\subsection*{Previous work}
Most results on the positive zeros of fewnomial systems are upper bounds \cite{Kh,BS,BD,BDF,PartI}. Soprunova and Sottile \cite{SoSo,SoSo2} obtain lower bounds for the number of real solutions as the degree of a projection of a real toric variety. For their Wronski systems, whose coefficient matrices are weighted colouring matrices of a balanced triangulation, they compute this degree, as well as the degree on the positive orthant \cite[Corollary~3.10]{SoSo}, as the signature of the triangulation; this is the special case of Theorem~\ref{thm:viro} for such matrices. Bihan, Santos and Spaenlehauer \cite{BSS} count the positive solutions of Viro systems with positively decorated simplices, and they introduce the signs $s(\sigma)$ \cite[Section~5]{BSS}. Theorem~\ref{thm:viro} shows that the signed count $\sum s(\sigma)$ is the degree of $F$, and thus does not depend on the Viro deformation. Bihan, Dickenstein and Giaroli \cite{BDG} combine degree theory and Gale duality to guarantee the existence of positive solutions. Finally, in reaction network theory, the Brouwer degree on the positive orthant is used to count equilibria \cite{CHW} and to detect multistationarity \cite{CFMW}.

\subsection*{Organization}
In Section~\ref{sec:degree} we define the degree, compute it on the Newton polytope and prove the vertex reduction. Section~\ref{sec:viro} describes the positive zeros of Viro systems and proves the Viro formula. In Section~\ref{sec:newton} we prove Theorems~\ref{thm:A}, \ref{thm:Ap} and~\ref{thm:comb}, the oriented matroid invariance and the half signed count of Proposition~\ref{prop:ends-newton}, which we interpret in terms of the curve defined by the first $n-1$ components of $F$. We also treat the cases $n\le2$ there, and give a piecewise linear proof of the Viro formula. Section~\ref{sec:coning} interprets Theorem~\ref{thm:A} as a Viro formula. In Section~\ref{sec:gale} we prove Theorems~\ref{thm:B}, \ref{thm:Bp} and~\ref{thm:F} and the half signed count of Proposition~\ref{prop:ends-gale}, we describe the Viro deformation and the vertex reduction on the Gale side, and we treat the cases $k\le2$. Section~\ref{sec:descartes} contains the Descartes rule of signs (Theorem~\ref{thm:descartes-intro}), Section~\ref{sec:bound} the proofs of Theorems~\ref{thm:C} and~\ref{thm:E}, and Section~\ref{sec:h} the H-duality and the proof of Theorem~\ref{thm:D}.

\subsection*{Genesis of the paper and use of AI}
The main idea of this paper comes from \cite{PartI}, where a bound for the topological degree was obtained on the Gale side. The formulas for the degree obtained through the compactified map $\hat F$ on the boundary of the Newton polytope were found with the help of an AI assistant. The Viro formula and its links with previous works, notably with the positively decorated simplices of Bihan, Santos and Spaenlehauer \cite{BSS}, were obtained by the author, as well as the link with the H-duality of Bihan, Dickenstein and Forsg\aa rd \cite{BDF24}, which gives the symmetry of the degree with respect to $n$ and $k$. The AI assistant also helped in several technical proofs, in the verification of the examples, and in the writing of the paper. The AI assistant used was Claude Opus 5.5 (Anthropic). The author has checked all the mathematical content of the paper, and takes full responsibility for it.

\subsection*{Data availability}
Data sharing is not applicable to this article, as no datasets were generated or analysed during the current study.

\subsection*{Conflict of interest}
The author has no conflicts of interest to declare that are relevant to the content of this article.

\section{The degree}\label{sec:degree}

In this section $\A$ is any finite set which affinely spans $\R^n$, and $C$ is any real $n\times N$ matrix. The genericity assumptions (U$_A$) and (U$_C$) are only used at the end. We define the degree of $F$, we show how to compute it on the Newton polytope $P=\conv(\A)$ through the moment map, and we prove the vertex reduction.

Recall that if $U\subset\R^n$ is open and $G\colon U\to\R^n$ is continuous with a compact zero set $G^{-1}(0)$, then the degree $\deg(G,U,0)$ is the Brouwer degree of $G$ with respect to $0$ on any bounded open set $V$ such that $G^{-1}(0)\subset V\subset\overline V\subset U$. By the excision property, it does not depend on $V$ (see for instance \cite{Lloyd}). If $G$ is $C^1$ and all its zeros are nondegenerate, then they are finitely many and $\deg(G,U,0)=\sum_{G(x)=0}\sign\det DG(x)$. Moreover, the degree does not change under continuous deformations $G_s$, $s\in[0,1]$, whose zeros remain in a fixed compact subset of $U$.

\begin{definition}\label{def:deg}
Assume that the positive zeros of $F$ form a compact subset of $\R^n_{>0}$. The \emph{degree} of the fewnomial map $F$ is
\[
\deg F=\deg\big(F,\R^n_{>0},0\big).
\]
\end{definition}

\subsection{The moment map and the compactified map}
For $x\in\R^n_{>0}$ we put
\[
\pi_i(x)=\frac{x^{a_i}}{\sum_{j=1}^Nx^{a_j}}\quad(i\in[N]),\qquad\mu(x)=\sum_{i=1}^N\pi_i(x)\,a_i .
\]
The weights $\pi_i(x)$ are positive with sum $1$. Thus $\pi(x)=(\pi_i(x))_i$ lies in the open simplex $\Delta_{N-1}^\circ$, and $\mu(x)$ lies in the interior of $P$. For a face $G$ of $P$, we identify $G$ with the set of indices $i$ such that $a_i\in G$. We write $F_G(x)=\sum_{i\in G}C_ix^{a_i}$ for the truncation of $F$ to $G$, and $\pi^G$, $\mu_G$ for the weights and the moment map of the point configuration $\{a_i\}_{i\in G}$, which are defined in the same way.

\begin{proposition}\label{prop:moment}
\begin{enumerate}
\item The moment map $\mu$ is an orientation preserving diffeomorphism from $\R^n_{>0}$ onto the interior of $P$.
\item There is a unique continuous map $\hat\pi\colon P\to\Delta_{N-1}$ such that $\hat\pi(\mu(x))=\pi(x)$ for $x\in\R^n_{>0}$. It satisfies $\sum_i\hat\pi_i(q)a_i=q$ for every $q\in P$, and for $q$ in the relative interior of a proper face $G$ we have $\hat\pi_i(q)=0$ for $i\notin G$ and $(\hat\pi_i(q))_{i\in G}=\pi^G(y)$ for every $y\in\R^n_{>0}$ with $\mu_G(y)=q$.
\end{enumerate}
\end{proposition}

This is a classical result. The closure of $\pi(\R^n_{>0})$ in the simplex is the positive part of the projective toric variety of $\A$, and $(p_i)\mapsto\sum_ip_ia_i$ is a homeomorphism from it onto $P$ (see \cite[Section~4.2]{Fulton-toric} for integer exponents and \cite{Sottile-moment} for the real points). In statistics, it is known as Birch's theorem. We include a short proof, since we allow real exponents.

\begin{proof}
1. In the coordinates $w=\log x$, $\mu$ is the gradient of the function $\ell(w)=\log\sum_ie^{\langle a_i,w\rangle}$. The Hessian of $\ell$ is the covariance matrix of the points $a_i$ with the weights $\pi_i$, which is positive definite since $\A$ affinely spans $\R^n$. Thus $\ell$ is strictly convex, and $\mu$ is an injective local diffeomorphism which preserves orientations. Its image is open and convex, and it is contained in the interior of $P$. Let us show that it is the whole interior. Let $q$ be in the interior of $P$. Since $\max_i\langle a_i,w\rangle-\langle q,w\rangle\ge c|w|$ with $c>0$, the strictly convex function $\ell(w)-\langle q,w\rangle$ tends to $+\infty$ at infinity. Hence it attains its minimum, at a point $w$ such that $\mu(e^w)=q$.

2. For a facet $\Gamma$ of $P$, let $\lambda_\Gamma$ be the affine function which vanishes on $\Gamma$ and is positive on $\A\setminus\Gamma$, and put $\kappa=\min_\Gamma\min_{i\notin\Gamma}\lambda_\Gamma(a_i)>0$. For $q=\mu(x)$ we have $\lambda_\Gamma(q)=\sum_{i\notin\Gamma}\pi_i(x)\lambda_\Gamma(a_i)\ge\kappa\sum_{i\notin\Gamma}\pi_i(x)$. Now let $q_m=\mu(x_m)$ tend to a point $q$ of the relative interior of a proper face $G$. Since $G$ is the intersection of the facets which contain it, the weights $\pi_i(x_m)$ with $i\notin G$ tend to $0$. The normalized weights $\big(\pi_i(x_m)/\sum_{j\in G}\pi_j(x_m)\big)_{i\in G}$ are equal to $\pi^G(x_m)$, and their barycenter $\mu_G(x_m)$ tends to $q$. Applying~1 to the configuration $\{a_i\}_{i\in G}$ in its affine span, we see that $\mu_G$ is a homeomorphism from the quotient of $\R^n_{>0}$ by the directions orthogonal to $G$ onto the relative interior of $G$. Moreover, $\pi^G$ factors through this quotient. Therefore $\pi^G(x_m)$ tends to $\pi^G(y)$, where $\mu_G(y)=q$. This defines $\hat\pi$ on $P$, with the stated properties. The same argument, applied to sequences in the relative interiors of the faces and combined with an induction on the dimension of the faces, shows that $\hat\pi$ is continuous. Uniqueness follows from the density of the interior.
\end{proof}

We define the \emph{compactified map}
\[
\hat F\colon P\to\R^n,\qquad\hat F(q)=\sum_{i=1}^N\hat\pi_i(q)\,C_i .
\]
It is continuous, and Proposition~\ref{prop:moment} gives the following two properties:
\begin{itemize}
\item on the interior of $P$, $\hat F(\mu(x))=F(x)/\sum_jx^{a_j}$, so the zeros of $\hat F$ in the interior of $P$ are the images under $\mu$ of the positive zeros of $F$;
\item on the relative interior of a proper face $G$, $\hat F(\mu_G(y))=F_G(y)/\sum_{j\in G}y^{a_j}$, so $\hat F$ vanishes on the relative interior of $G$ if and only if the truncation $F_G$ has a zero in $\R^n_{>0}$.
\end{itemize}

\subsection{The degree at infinity}
We use the compactified map $\hat F$ to compute $\deg F$, and to control its behaviour when the coefficients vary. We say that $F$ is \emph{nondegenerate at infinity} if the following condition holds:
\begin{itemize}
\item[(N)] for every proper face $G$ of $P$, the truncation $F_G$ has no zero in $\R^n_{>0}$.
\end{itemize}
Equivalently, $\hat F$ does not vanish on $\partial P$.

\begin{proposition}\label{prop:degree}
Assume (N).
\begin{enumerate}
\item The positive zeros of $F$ form a compact subset of $\R^n_{>0}$, so $\deg F$ is defined, and $\deg F=\deg(\hat F,P^\circ,0)$, where $P^\circ$ is the interior of $P$. Consequently, $\deg F$ is the degree of the map $\hat F/|\hat F|\colon\partial P\to S^{n-1}$, where $\partial P$ is oriented as the boundary of $P$. If the positive zeros are all nondegenerate, then $\deg F=\sum_{F(x)=0}\sign\det DF(x)$; in particular the number of positive zeros is at least $|\deg F|$, and congruent to $\deg F$ modulo $2$.
\item If $\deg F\neq0$, then $F$ has at least one positive zero.
\item The degree is constant along any continuous family of coefficient matrices, parametrized by a connected space, for which (N) holds.
\end{enumerate}
\end{proposition}

\begin{proof}
1. The zeros of $\hat F$ form a closed subset of $P$ which does not meet $\partial P$, hence a compact subset of the interior of $P$. Recall that $\mu$ is a homeomorphism from $\R^n_{>0}$ onto the interior of $P$, and that $\hat F(\mu(x))=F(x)/\rho(x)$ with $\rho(x)=\sum_jx^{a_j}>0$. Therefore the positive zeros of $F$ form a compact set $Z=\mu^{-1}(\hat F^{-1}(0))$. Let $V$ be a bounded open set such that $Z\subset V\subset\overline V\subset\R^n_{>0}$. The maps $F/((1-s)+s\rho)$, $s\in[0,1]$, have the same zeros as $F$, so $\deg(F,V,0)=\deg(F/\rho,V,0)$. Since $F/\rho=\hat F\circ\mu$ and $\mu$ is an orientation preserving diffeomorphism (Proposition~\ref{prop:moment}), we get $\deg(F/\rho,V,0)=\deg(\hat F,\mu(V),0)$, and by excision this is $\deg(\hat F,P^\circ,0)$. Finally, $P$ is homeomorphic to a closed ball by a radial homeomorphism preserving orientations. Thus, by the Kronecker--Poincar\'e--Hopf index theorem (see \cite[\S6]{Milnor} or \cite{Lloyd}) applied to $\hat F$ on $P$, $\deg(\hat F,P^\circ,0)$ is the degree of $\hat F/|\hat F|\colon\partial P\to S^{n-1}$. The statement on nondegenerate zeros is the usual computation of the degree at a regular value. The consequences follow, since a sum of $m$ signs is at most $m$ in absolute value and congruent to $m$ modulo $2$.

2. This is a standard property of the degree.

3. The map $\hat\pi$ does not depend on $C$, so $\hat F$ depends continuously on $C$, uniformly on $P$. Along the family, $\hat F$ does not vanish on $\partial P$. Hence the maps $\hat F/|\hat F|$ on $\partial P$ form a homotopy, and their degree is constant. We conclude by~1.
\end{proof}

Observe that under (N), $\deg F$ only depends on the truncations of $F$ to the proper faces of $P$. In other words, the coefficients of the monomials $x^{a_i}$ with $a_i$ in the interior of $P$ can be changed arbitrarily without changing the degree. Under (H), the degree even only depends on the columns indexed by the vertices of $P$ (Proposition~\ref{lem:vertex}).

\begin{lemma}\label{lem:N}
\begin{enumerate}
\item If every proper face $G$ of $P$ is a simplex whose only points in $\A$ are its vertices, then (N) holds if and only if, for every proper face $G$, $\ker C^G$ contains no vector with positive coordinates. This is the case under (U$_A$).
\item Under (U$_A$) and (U$_C$), (N) holds, and it also holds for $C\,\mathrm{diag}(\lambda)$ for every $\lambda\in\R^N_{>0}$.
\item Consider the condition
\begin{itemize}
\item[(H)] for every proper face $G$ of $P$, the columns $C_i$, $i\in G$, lie in an open half-space of $\R^n$; equivalently, by Gordan's theorem, $\ker C^G$ contains no nonzero vector with nonnegative coordinates.
\end{itemize}
Then (H) implies (N), and it is preserved by the rescalings $C\mapsto C\,\mathrm{diag}(\lambda)$, $\lambda\in\R^N_{>0}$. If every proper face of $P$ is a simplex whose only points in $\A$ are its vertices, (H) is equivalent to (N).
\end{enumerate}
\end{lemma}

\begin{proof}
1. Under (U$_A$), the points of a proper face lie on an affine hyperplane. Hence there are at most $n$ of them, and they are affinely independent, so every proper face is a simplex whose only points in $\A$ are its vertices. Let $G$ be such a face. Up to the directions orthogonal to $G$, the map $y\mapsto(y^{a_i})_{i\in G}$ from $\R^n_{>0}$ reaches every positive vector up to a positive factor. Therefore $F_G$ has a positive zero if and only if $\ker C^G$ contains a positive vector.

2. By (U$_C$), the at most $n$ columns $C_i$, $i\in G$, are linearly independent, so $\ker C^G=0$. Multiplying the columns by positive numbers does not change this.

3. If $\langle w,C_i\rangle>0$ for $i\in G$, then $\langle w,F_G(y)\rangle>0$ for $y\in\R^n_{>0}$. Thus (H) implies (N). Moreover, positive rescalings of the columns preserve the half-space. In the simplicial case, a nonzero nonnegative vector of $\ker C^G$ is a positive vector of $\ker C^{G'}$, where $G'$ is the face of $G$ spanned by the points in its support. Hence (N) implies (H) by~1.
\end{proof}

\begin{proposition}\label{lem:vertex}
Assume (H). Let $V\subset[N]$ be the set of indices $i$ such that $a_i$ is a vertex of $P$, and $F_V(x)=\sum_{i\in V}C_ix^{a_i}$. Then $F_V$ satisfies (H), and $\deg F=\deg F_V$.
\end{proposition}

\begin{proof}
The Newton polytope of $F_V$ is $P$, and its columns on a face of $P$ are among those of $F$, so $F_V$ satisfies (H). For $s\in[0,1]$, let $F_s=F_V+s\sum_{i\notin V}C_ix^{a_i}$, with the exponent set $\A$. Let $G$ be a proper face of $P$, and let $w$ be the vector given by (H) for $F$ and $G$, so that $\langle w,C_i\rangle>0$ for all $i\in G$. The truncation $(F_s)_G$ has the columns $C_i$ for $i\in G\cap V$, which is a nonempty set since $G$ has a vertex, and the columns $sC_i$ for $i\in G\setminus V$, which satisfy $\langle w,sC_i\rangle\ge 0$. Hence $\langle w,(F_s)_G(y)\rangle>0$ for $y\in\R^n_{>0}$, and $F_s$ satisfies (N) for all $s\in[0,1]$. By Proposition~\ref{prop:degree}(3), we get $\deg F=\deg F_1=\deg F_0=\deg F_V$.
\end{proof}

\begin{remark}[Without the half-space condition]\label{rem:general}
\begin{enumerate}
\item \emph{Why (N).} Proposition~\ref{prop:degree} only requires (N), and this condition cannot be dropped. Indeed, if some truncation $F_G$ has a positive zero, then $\hat F$ vanishes on $\partial P$, and the positive zeros of $F$ may fail to form a compact set. Moreover, under small perturbations of the coefficients, zeros of $F$ escape to infinity in the direction of the face $G$.
\item \emph{Without (H), the degree is not determined by the oriented matroids.} Rescaling the columns of $C$ by positive numbers does not change the signs of the minors of $A$ and $C$. Under (H) it preserves (N), so the degree does not change; this is what Viro's construction uses. Without (H), however, a rescaling can make a truncation acquire a positive zero, and the degree can jump. For instance, let $n=2$, $\A=\{(0,0),(1,0),(2,0),(0,1)\}$, and for $\lambda>0$ consider
\[
F_\lambda(x,y)=\big(1-2\lambda x+x^2,\ -\lambda x+x^2-y\big),
\]
so that $C_1=(1,0)$, $C_2=\lambda(-2,-1)$, $C_3=(1,1)$, $C_4=(0,-1)$. The six $2\times2$ minors of $C$ are $-\lambda,1,-1,-\lambda,2\lambda,-1$. Thus (U$_C$) holds, and the signs of all minors do not depend on $\lambda$. The bottom edge of $P$ contains three points of $\A$, and the columns $(1,0)$, $(-2,-1)$, $(1,1)$ do not lie in an open half-plane, so (H) fails on this edge. The truncation to this edge is $(1-2\lambda x+x^2,\,-\lambda x+x^2)$. Its second component gives $x=\lambda$, and then the first one gives $1-\lambda^2=0$. Hence (N) holds for $\lambda\neq1$ and fails for $\lambda=1$. The positive zeros of $F_\lambda$ satisfy $y=x^2-\lambda x>0$, that is, $x>\lambda$, and $x^2-2\lambda x+1=0$. For $\lambda<1$ there is no real root, and $\deg F_\lambda=0$. For $\lambda>1$ the roots are $x_\pm=\lambda\pm\sqrt{\lambda^2-1}$. Since $x_+x_-=1$, we have $x_-<1<\lambda$, so only $x_+$ gives a positive zero. At this zero $\det DF_\lambda=-2(x_+-\lambda)<0$, and therefore $\deg F_\lambda=-1$. As $\lambda$ crosses $1$, a zero comes in from infinity through the bottom edge.
\end{enumerate}
\end{remark}

From now on, as in the introduction, we assume (U$_A$) and (U$_C$) unless otherwise stated. The general case, under (H), is treated in Sections~\ref{sec:viro} to~\ref{sec:gale}, where the assumptions are stated explicitly. Under (U$_A$) and (U$_C$), $\deg F$ is defined and does not change under the rescalings of the columns of $C$ by positive numbers (Lemma~\ref{lem:N}(2) and Proposition~\ref{prop:degree}(3)). Moreover, the face structure of $P$ is determined by the signs of the maximal minors of $A$.

\begin{corollary}\label{cor:parity}
Assume (U$_A$) and (U$_C$), and that all the positive zeros of $F$ are nondegenerate. Let $U$ be an upper bound for their number. Then their number is at least $|\deg F|$, and at most $U$ if $U\equiv\deg F\pmod2$, at most $U-1$ otherwise.
\end{corollary}

\begin{proof}
By Proposition~\ref{prop:degree}(1), the number of positive zeros is the number of terms in a sum of signs equal to $\deg F$. Hence it is at least $|\deg F|$, and congruent to $\deg F$ modulo $2$.
\end{proof}

By Theorem~\ref{cor:om}, $\deg F$ only depends on the oriented matroids of $A$ and $C$. Therefore its parity can be combined with any upper bound which is valid for given oriented matroids, such as the bounds of \cite{PartI}.

\section{Viro's construction}\label{sec:viro}

Throughout this section we assume (U$_C$) and the half-space condition (H) of Lemma~\ref{lem:N}(3). Both hold in the generic case. By Lemma~\ref{lem:N}(3), (H) implies (N), so that $\deg F$ is defined, and (H) still holds after any rescaling of the columns of $C$ by positive numbers. Our aim is to describe the positive zeros of the Viro deformations of $F$. For $h\in\R^N$ and $t>0$ we put
\[
F^{h,t}(x)=\sum_{i=1}^Nt^{h_i}C_i\,x^{a_i}.
\]
This is the fewnomial map with the same exponents and with coefficient matrix $C\,\mathrm{diag}(t^{h_1},\dots,t^{h_N})$. Its maximal minors are those of $C$ multiplied by positive numbers. Hence it satisfies (U$_C$) and (H), and by Proposition~\ref{prop:degree}(3) we have
\begin{equation}\label{eq:viro-invariance}
\deg F^{h,t}=\deg F\qquad\text{for all }t>0 .
\end{equation}

\subsection{Regular triangulations and cell systems}
Let $A$ be the $(n+1)\times N$ matrix with columns $(1,a_i)$. For $\sigma\subset[N]$, let $A^\sigma$ be its submatrix with the columns indexed by $\sigma$, in increasing order. We say that a set $\sigma\subset[N]$ with $|\sigma|=n+1$ is a \emph{maximal simplex of the regular subdivision $T_h$} induced by the heights $h$ if there is an affine function $\lambda_\sigma$ on $\R^n$ such that
\[
\lambda_\sigma(a_i)=h_i\ \ (i\in\sigma),\qquad\lambda_\sigma(a_i)<h_i\ \ (i\notin\sigma).
\]
In other words, $\{(a_i,h_i)\}_{i\in\sigma}$ spans a lower facet of the convex hull of the lifted points. We say that $h$ is \emph{generic} if every lower facet contains exactly $n+1$ lifted points. In this case the simplices $\conv(a_i,\ i\in\sigma)$, $\sigma\in T_h$, form a triangulation of $\conv(\A)$. Note that under (U$_A$), every facet of $\conv(\A)$ contains exactly $n$ points of $\A$, so it is a face of exactly one simplex of $T_h$.

Recall from \cite[Definition~3.1]{BSS} that a real $n\times(n+1)$ matrix $M$ is \emph{positively spanning} if the numbers $(-1)^j\operatorname{minor}(M,j)$, $j=1,\dots,n+1$, are nonzero and have the same sign, where $\operatorname{minor}(M,j)$ is the determinant of the matrix obtained by deleting the $j$-th column of $M$. When $M$ has rank $n$, these numbers are, by Cramer's rule, the coordinates of a vector spanning $\ker M$. Thus $M$ is positively spanning if and only if $\ker M$ is a line spanned by a vector with positive coordinates, that is, if and only if the columns of $M$ positively span $\R^n$ \cite[Proposition~3.2]{BSS}.

For $\sigma\in T_h$, the \emph{cell system} is the restriction $F_\sigma(x)=\sum_{l\in\sigma}C_lx^{a_l}$. Following \cite{BSS}, we say that $\sigma$ is \emph{positively decorated} by $C$ if the $n\times(n+1)$ matrix $C^\sigma$ is positively spanning. Let $\widehat C^\sigma$ be the $(n+1)\times(n+1)$ matrix with columns $(1,C_l)$, $l\in\sigma$, in increasing order, and put
\[
s(\sigma)=\sign\det A^\sigma\cdot\sign\det\widehat C^\sigma .
\]
The two factors are the orientations of the simplex $\sigma$ induced by the points $a_l$ and by the decoration $C_l$, $l\in\sigma$. This sign was introduced by Bihan, Santos and Spaenlehauer \cite[Section~5]{BSS}, and we use their notion of positively decorated simplices throughout. Expanding $\det\widehat C^\sigma$ along its first row gives $\det\widehat C^\sigma=\sum_j(-1)^{1+j}\operatorname{minor}(C^\sigma,j)$. Therefore, when $\sigma$ is positively decorated, $\sign\det\widehat C^\sigma=\sign\operatorname{minor}(C^\sigma,1)$.

For a $C^1$ map $G=(G_1,\dots,G_n)$ defined on $\R^n_{>0}$, or on $(\mathbb C^*)^n$, we use the \emph{toric Jacobian matrix}
\[
D^{\mathrm t}G(x)=\Big(x_j\frac{\partial G_i}{\partial x_j}(x)\Big)_{i,j}=DG(x)\,\mathrm{diag}(x_1,\dots,x_n).
\]
For $x\in\R^n_{>0}$, $\det D^{\mathrm t}G(x)=x_1\cdots x_n\det DG(x)$ has the sign of $\det DG(x)$. Note the following simple facts. For a monomial, $D^{\mathrm t}(x^b)=x^b\,b^T$. The chain rule $D^{\mathrm t}(L\circ\phi)=DL(\phi)\,D^{\mathrm t}\phi$ holds for every $C^1$ map $L$. Finally, if $\gamma\in\R^n_{>0}$ and $G_\gamma(x)=G(\gamma_1x_1,\dots,\gamma_nx_n)$, then $D^{\mathrm t}G_\gamma(x)=D^{\mathrm t}G(\gamma_1x_1,\dots,\gamma_nx_n)$. The following lemma is \cite[Proposition~3.3]{BSS}, together with the sign of the Jacobian determinant.

\begin{lemma}\label{lem:cell}
Let $\sigma\in T_h$. The cell system $F_\sigma$ has at most one positive zero, and it has one if and only if $\sigma$ is positively decorated. In this case the positive zero $z_\sigma$ is nondegenerate, and $\sign\det DF_\sigma(z_\sigma)=\sign\det D^{\mathrm t}F_\sigma(z_\sigma)=s(\sigma)$.
\end{lemma}

\begin{proof}
Write $\sigma=\{l_0<l_1<\dots<l_n\}$. As in \cite{BSS}, we multiply $F_\sigma$ by $x^{-a_{l_0}}$. The system $F_\sigma=0$ is then equivalent on $\R^n_{>0}$ to
\[
x^{-a_{l_0}}F_\sigma(x)=C_{l_0}+\sum_{j=1}^nC_{l_j}\,x^{a_{l_j}-a_{l_0}}=0 .
\]
Since $A^\sigma$ is invertible, the vectors $a_{l_j}-a_{l_0}$, $j=1,\dots,n$, form a basis of $\R^n$. Hence the monomial map $\phi\colon x\mapsto Y=(x^{a_{l_1}-a_{l_0}},\dots,x^{a_{l_n}-a_{l_0}})$ is a diffeomorphism of $\R^n_{>0}$. Under $\phi$, the system becomes the linear system $L(Y)=C_{l_0}+\sum_jY_jC_{l_j}=C^\sigma\,(1,Y)=0$, $Y\in\R^n_{>0}$. Its solutions are the vectors $Y>0$ such that $(1,Y)\in\ker C^\sigma$. By (U$_C$), $\ker C^\sigma$ is a line. Therefore there is at most one solution, and there is one if and only if this line is spanned by a positive vector, that is, if and only if $C^\sigma$ is positively spanning.

Assume now that $\sigma$ is positively decorated. Since $F_\sigma(z_\sigma)=0$, the toric Jacobian matrix of $x^{-a_{l_0}}F_\sigma$ at $z_\sigma$ is $z_\sigma^{-a_{l_0}}D^{\mathrm t}F_\sigma(z_\sigma)$. On the other hand, $x^{-a_{l_0}}F_\sigma=L\circ\phi$, so by the chain rule this matrix is also $DL\cdot D^{\mathrm t}\phi(z_\sigma)=C^{\sigma\setminus\{l_0\}}\,\mathrm{diag}(Y)\,M$, where $M$ is the $n\times n$ matrix with rows $a_{l_j}-a_{l_0}$. Here $\det C^{\sigma\setminus\{l_0\}}=\operatorname{minor}(C^\sigma,1)$, which is nonzero and has the sign of $\det\widehat C^\sigma$. Moreover $\det M=\det A^\sigma$, as one sees by subtracting the first column of $A^\sigma$ from the others. Since $Y$ and $z_\sigma$ are positive, we conclude that $\det D^{\mathrm t}F_\sigma(z_\sigma)\neq0$ and that its sign is $s(\sigma)$, which is also the sign of $\det DF_\sigma(z_\sigma)$.
\end{proof}

\subsection{The Viro formula}

In the following theorem, assertion~(1) describes the positive zeros of $F^{h,t}$ for $t$ small. They are the zeros constructed by Bihan, Santos and Spaenlehauer \cite[Theorem~3.4]{BSS} near the positively decorated simplices, with the Jacobian signs $s(\sigma)$, and an elementary localization argument shows that there are no others. Assertion~(2) is the \emph{Viro formula} for the degree. It follows from~(1), but it also has a short piecewise linear proof which does not use the zeros of $F^{h,t}$ (Proposition~\ref{prop:viro-pl}).

\begin{theorem}\label{thm:viro}
Assume (U$_C$) and (H), for instance (U$_A$) and (U$_C$), and let $h\in\R^N$ be generic.
\begin{enumerate}
\item For $t>0$ small enough, the positive zeros of $F^{h,t}$ are nondegenerate, and there is exactly one of them for each positively decorated simplex $\sigma\in T_h$, and no other. The sign of the Jacobian determinant at the zero attached to $\sigma$ is $s(\sigma)$.
\item Consequently
\[
\deg F=\sum_{\substack{\sigma\in T_h\\ \sigma\ \text{positively decorated}}}s(\sigma).
\]
\end{enumerate}
\end{theorem}

\begin{proof}
Assertion~2 is Proposition~\ref{prop:viro-pl}. It also follows from assertion~1, Proposition~\ref{prop:degree}(1) and~\eqref{eq:viro-invariance}. Indeed, the degree of $F$ is the degree of $F^{h,t}$ for any $t>0$, and for $t$ small this is the sum of the Jacobian signs at the positive zeros of $F^{h,t}$.

Let us prove assertion~1, following the proof of \cite[Theorem~3.4]{BSS}. For $\sigma\in T_h$, the function $\lambda_\sigma$ is affine, and we write
\[
\lambda_\sigma(a)=\langle\alpha_\sigma,a\rangle+\beta_\sigma,\qquad\alpha_\sigma\in\R^n,\ \beta_\sigma\in\R .
\]
For $x\in\R^n_{>0}$ we put $x\,t^{-\alpha_\sigma}=(x_1t^{-\alpha_{\sigma,1}},\dots,x_nt^{-\alpha_{\sigma,n}})$, and for a set $K$ of such points we put $K\,t^{-\alpha_\sigma}=\{x\,t^{-\alpha_\sigma}:\ x\in K\}$. Note that the vectors $\alpha_\sigma$, $\sigma\in T_h$, are pairwise distinct, since two distinct lower facets of the lifted configuration $\{(a_i,h_i)\}$ cannot be parallel.

\emph{Step 1: the zeros near the positively decorated simplices.} Let $\sigma\in T_h$. Since $t^{h_i}(x\,t^{-\alpha_\sigma})^{a_i}=t^{h_i-\langle\alpha_\sigma,a_i\rangle}x^{a_i}$, we get
\begin{equation}\label{eq:local}
\frac{F^{h,t}(x\,t^{-\alpha_\sigma})}{t^{\beta_\sigma}}=F_\sigma(x)+r^{(\sigma)}_t(x),\qquad r^{(\sigma)}_t(x)=\sum_{i\notin\sigma}t^{h_i-\lambda_\sigma(a_i)}C_i\,x^{a_i}.
\end{equation}
The exponents $h_i-\lambda_\sigma(a_i)$, $i\notin\sigma$, are positive since $\sigma\in T_h$. Therefore, on every compact set, $r^{(\sigma)}_t$ tends to $0$ in the $C^1$ topology as $t\to0$.

Assume that $\sigma$ is positively decorated, and let $z_\sigma$ be the positive zero of the cell system $F_\sigma$ (Lemma~\ref{lem:cell}; this is \cite[Proposition~3.3]{BSS}). Since $z_\sigma$ is nondegenerate, there is a compact ball $K_\sigma\subset\R^n_{>0}$ centred at $z_\sigma$ on which $F_\sigma$ is a diffeomorphism onto its image. Then, for $t$ small, $F_\sigma+r^{(\sigma)}_t$ has exactly one zero $z_{\sigma,t}$ in $K_\sigma$. This zero is nondegenerate, and the sign of its toric Jacobian determinant is that of $F_\sigma$ at $z_\sigma$, that is, $s(\sigma)$ by Lemma~\ref{lem:cell}. By~\eqref{eq:local}, $z_{\sigma,t}\,t^{-\alpha_\sigma}$ is a nondegenerate positive zero of $F^{h,t}$. The toric Jacobian matrix commutes with the scaling $x\mapsto x\,t^{-\alpha_\sigma}$, and $t^{\beta_\sigma}>0$. Hence the toric Jacobian determinant of $F^{h,t}$ at $z_{\sigma,t}\,t^{-\alpha_\sigma}$ has the sign $s(\sigma)$, and so has its Jacobian determinant. This is the construction of \cite[Theorem~3.4]{BSS}, and the sign is the sign $s(\sigma)$ of \cite[Section~5]{BSS}.

\emph{Step 2: there are no other zeros.} We show that for $t$ small every positive zero of $F^{h,t}$ lies in one of the sets $K_\sigma\,t^{-\alpha_\sigma}$ of Step~1, where $\sigma$ is positively decorated. By Step~1, it is then the zero $z_{\sigma,t}\,t^{-\alpha_\sigma}$. Put $s=-\log t$ and write a positive point as $x=t^{v}=e^{-sv}$. Then $t^{h_i}x^{a_i}=e^{-s\varphi_i(v)}$ with $\varphi_i(v)=h_i+\langle a_i,v\rangle$. We put $\varphi(v)=\min_i\varphi_i(v)$, and for $\gamma>0$ we let $J_\gamma(v)=\{i:\ \varphi_i(v)<\varphi(v)+\gamma\}$. Observe that the sets $\{i:\ \varphi_i(v)=\varphi(v)\}$ are the cells of $T_h$.

Let $J\subset[N]$ be a subset which is not contained in a simplex of $T_h$. The function $v\mapsto\max_{i\in J}\varphi_i(v)-\varphi(v)$ is convex, piecewise affine and positive, so its minimum is positive. Let $\gamma>0$ be smaller than all these minima. Then $J_\gamma(v)$ is contained in a simplex of $T_h$ for every $v$. Dividing the equation $F^{h,t}(x)=0$ by $e^{-s\varphi(v)}$, we get
\[
\sum_{i\in J_\gamma(v)}\rho_iC_i=-\sum_{i\notin J_\gamma(v)}\rho_iC_i,\qquad \rho_i=e^{-s(\varphi_i(v)-\varphi(v))}\in(0,1].
\]
The right-hand side has norm at most $N\max_i|C_i|\,e^{-s\gamma}$. On the left-hand side, the coefficients lie in $[e^{-s\gamma},1]$, and one of them is equal to $1$. If $|J_\gamma(v)|\le n$, the columns $C_i$, $i\in J_\gamma(v)$, are linearly independent by (U$_C$). If $J_\gamma(v)$ is a simplex of $T_h$ which is not positively decorated, then $\ker C^\sigma$ contains no nonzero nonnegative vector. In both cases the left-hand side has norm at least a constant $c>0$. Consequently, for $s$ large, every positive zero $x$ of $F^{h,t}$ satisfies $J_\gamma(v)=\sigma$ for a positively decorated simplex $\sigma$, and the vector $(\rho_l)_{l\in\sigma}$ is close to the positive vector spanning $\ker C^\sigma$ with largest coordinate $1$. Put $y=x\,t^{\alpha_\sigma}$, so that $x=y\,t^{-\alpha_\sigma}$. The ratios $y^{a_l-a_{l'}}=\rho_l/\rho_{l'}$, $l,l'\in\sigma$, are then close to the corresponding ratios at $z_\sigma$. Since the vectors $a_l-a_{l'}$ span $\R^n$, it follows that $y$ lies in $K_\sigma$ for $t$ small.
\end{proof}

\begin{remark}[Complex zeros]\label{rem:complex}
Assume (U$_A$) and that the exponents are integers. The argument of Step~2, applied to the moduli of the coordinates of the complex zeros as in the proof of Bernstein's theorem by Huber and Sturmfels \cite{HS}, gives three counts for $t$ small. A simplex $\sigma\in T_h$ carries $|\det A^\sigma|$ zeros of $F^{h,t}$ in $(\mathbb C^*)^n$. Among them, one is positive if $\sigma$ is positively decorated and none otherwise, and this positive zero has the Jacobian sign $s(\sigma)$. Summing over $T_h$, we get the number $\mathrm{Vol}(\A)$ of complex zeros, the number of positive zeros of the Viro system, and the degree $\deg F$. The first and the last numbers do not depend on $h$, while the second one does. In particular, Viro systems attain the Kushnirenko bound with nondegenerate zeros. For different supports, the same argument, with a different lifting for each equation and mixed subdivisions instead of triangulations, gives Bernstein's theorem, as in \cite{HS}.
\end{remark}

\begin{example}[A running example: the Viro formula]\label{ex:run}
We use the map of the introduction as a running example:
\[
F(x,y)=\big(3x-3x^3+x^4y^2+x^2y^3-3y^2,\ 2x^3-3x^4y^2+3x^2y^3-2y^2\big),
\]
with $\A=\{a_1,\dots,a_5\}=\{(1,0),(3,0),(4,2),(2,3),(0,2)\}$, the vertices of a pentagon $P$ in counterclockwise order, and columns
\[
C_1=(3,0),\quad C_2=(-3,2),\quad C_3=(1,-3),\quad C_4=(1,3),\quad C_5=(-3,-2).
\]
Here $n=k=2$ and $N=5$, and (U$_A$) and (U$_C$) hold. A triangle $\sigma\subset[5]$ is positively decorated if and only if its three columns positively span $\R^2$. Among the ten triangles, these are exactly the five \emph{ears} $\{1,2,3\}$, $\{2,3,4\}$, $\{3,4,5\}$, $\{1,4,5\}$, $\{1,2,5\}$ formed by three consecutive vertices, and all of them have $s(\sigma)=+1$. The triangulations of $P$ are the five fans from a vertex. They are regular, and each of them contains exactly two ears. For instance, $h=(-1,0,0,0,0)$ induces the fan $T_h=\{\{1,2,3\},\{1,3,4\},\{1,4,5\}\}$ from $a_1$, in which $\{1,2,3\}$ and $\{1,4,5\}$ are positively decorated. By Theorem~\ref{thm:viro}, we get $\deg F=1+1=2$. Moreover, for $t$ small the Viro map $F^{h,t}$ has exactly two positive zeros, both with a positive Jacobian determinant.
\end{example}

\begin{example}[Cancellation]\label{ex:cancel}
In general the signs $s(\sigma)$ do not all agree. Let $\A=\{0,1,2\}^2$, with $a_1=(0,0)$, $a_2=(1,0)$, $a_3=(2,0)$, $a_4=(0,1)$, \dots, $a_9=(2,2)$, and consider
\[
\begin{aligned}
f_1&=-1-x-x^2+2y-3xy+x^2y-2y^2+2xy^2+3x^2y^2,\\
f_2&=1+3x-3x^2-y-xy+x^2y+3y^2+xy^2-x^2y^2,
\end{aligned}
\]
so that $C_1=(-1,1)$, $C_2=(-1,3)$, $C_3=(-1,-3)$, $C_4=(2,-1)$, $C_5=(-3,-1)$, $C_6=(1,1)$, $C_7=(-2,3)$, $C_8=(2,1)$, $C_9=(3,-1)$. Here $n=2$ and $k=6$. Condition (U$_A$) fails, since each edge of the square contains three points of $\A$. However, (U$_C$) holds, and so does (H): on each edge, the three columns lie in an open half-plane (for the edge $\{1,2,3\}$, their first coordinates are $-1$). The heights $h_i=\alpha^2+\beta^2-\alpha\beta$ for $a_i=(\alpha,\beta)$, that is, $h=(0,1,4,1,1,3,4,3,4)$, induce the unimodular triangulation
\[
T_h=\big\{\{1,2,5\},\{1,4,5\},\{2,3,6\},\{2,5,6\},\{4,5,8\},\{4,7,8\},\{5,6,9\},\{5,8,9\}\big\}.
\]
Five of its triangles are positively decorated: $\{1,4,5\}$, $\{2,3,6\}$ and $\{5,8,9\}$, with $s(\sigma)=+1$, and $\{4,5,8\}$ and $\{5,6,9\}$, with $s(\sigma)=-1$. By Theorem~\ref{thm:viro}, we get $\deg F=3-2=1$. On the other hand, for $t$ small the Viro map $F^{h,t}$ has five nondegenerate positive zeros, three with a positive and two with a negative Jacobian determinant. For the heights $h_i=\alpha^2+\beta^2+\alpha\beta$, the triangulation is the mirror image, and the positively decorated triangles are $\{2,4,5\}$ and $\{3,5,6\}$, with sign $+1$, and $\{4,5,7\}$, with sign $-1$. Thus the number of positive zeros of the Viro map depends on $h$, while the degree does not. In agreement with the vertex reduction (Proposition~\ref{lem:vertex}), we have $\deg F=\deg F_V$ with $F_V=C_1+C_3x^2+C_7y^2+C_9x^2y^2$. Indeed, the closed polygon $C_1C_3C_9C_7$ of the columns at the vertices of the square, taken in counterclockwise order, turns once around the origin (Corollary~\ref{cor:n2}). Finally, since the Viro sum has an odd number of terms, $F$ has a positive zero (Theorem~\ref{thm:F}).
\end{example}

\section{The Newton polytope formula}\label{sec:newton}

Theorems~\ref{thm:A} and~\ref{thm:Ap}, as well as the Viro formula of Theorem~\ref{thm:viro}(2), all reduce to the computation of the degree of a piecewise linear map. For a triangulation $T$ of $\partial P$ or of $P$ with vertices in $\A$, we denote by $\Pi_T$ the continuous map which is affine on every simplex of $T$ and sends each vertex $a_s$ of $T$ to $C_s$.

\begin{lemma}\label{lem:homotopy}
Assume (H), and let $T$ be a triangulation of $\partial P$ or of $P$ with vertices in $\A$ such that every proper face of $P$ is a union of simplices of $T$. Then $\Pi_T$ does not vanish on $\partial P$, and $\deg F$ is the degree of $\Pi_T/|\Pi_T|\colon\partial P\to S^{n-1}$, where $\partial P$ is oriented as the boundary of $P$.
\end{lemma}

\begin{proof}
By Lemma~\ref{lem:N}(3), (H) implies (N). Hence, by Proposition~\ref{prop:degree}(1), $\deg F$ is the degree of $\hat F/|\hat F|$ on $\partial P$. Let $q$ be a point of the relative interior of a proper face $G$ of $P$. The simplex of $T$ whose relative interior contains $q$ lies in $G$, so $\Pi_T(q)=\sum_{i\in G}\beta_i(q)C_i$, where the $\beta_i(q)$ are the barycentric coordinates in this simplex, extended by $0$. On the other hand, $\hat F(q)=\sum_{i\in G}\hat\pi_i(q)C_i$ with nonnegative weights, not all zero, by Proposition~\ref{prop:moment}(2). Thus every convex combination $(1-\theta)\hat F(q)+\theta\,\Pi_T(q)$, $\theta\in[0,1]$, is a nonnegative nonzero combination of the columns $C_i$, $i\in G$, and it does not vanish by (H). Therefore $\hat F/|\hat F|$ and $\Pi_T/|\Pi_T|$ are homotopic on $\partial P$.
\end{proof}

\begin{lemma}\label{lem:PL}
Let $T$ be a triangulation of $\partial P$ with vertices in $\A$ such that every face of $P$ is a union of simplices of $T$, put $\Pi=\Pi_T$, assume that $\Pi$ does not vanish, and let $u$ be as in Theorem~\ref{thm:A}. Then the degree of $\Pi/|\Pi|\colon\partial P\to S^{n-1}$ is given by the formula of Theorem~\ref{thm:Ap}.
\end{lemma}

\begin{proof}
We compute the degree as the number of preimages of the direction $u$, counted with signs. Let $\tau=\{s_1<\dots<s_n\}$ be a simplex of $T$ of dimension $n-1$, contained in the facet $\Gamma$ of $P$. The restriction of $\Pi$ to $\tau$ is $\Pi(\sum_r\beta_ra_{s_r})=\sum_r\beta_rC_{s_r}$, with $\beta_r\ge0$ and $\sum_r\beta_r=1$. It is a positive multiple of $u$ only if $u\in\cone(C_s,\ s\in\tau)$. If the columns $C_s$, $s\in\tau$, are linearly dependent, then by Carath\'eodory's theorem this cone is the union of the cones spanned by $n-1$ of them. Hence it does not contain $u$, and $\tau$ contributes nothing. Otherwise $C^\tau$ is invertible. Then $\Pi$ is a positive multiple of $u$ at exactly one point of $\tau$ if $u\in\cone(C_s,\ s\in\tau)$, and at no point otherwise. By the choice of $u$, this point lies in the relative interior of $\tau$.

We orient $\tau$ by the basis $(a_{s_2}-a_{s_1},\dots,a_{s_n}-a_{s_1})$ of its direction. Let $q$ be a point such that $\Pi(q)\in\R_{>0}u$. For this orientation of $\tau$, and for the orientation of $S^{n-1}$ as a boundary, the local degree of $\Pi/|\Pi|$ at $q$ is the sign of
\[
\det\big(\Pi(q),\,C_{s_2}-C_{s_1},\dots,C_{s_n}-C_{s_1}\big)=\det\big(C_{s_1},C_{s_2}-C_{s_1},\dots,C_{s_n}-C_{s_1}\big)=\det C^\tau,
\]
where we used $\Pi(q)=C_{s_1}+\sum_{r\ge2}\beta_r(C_{s_r}-C_{s_1})$. It remains to compare this orientation of $\tau$ with its orientation as a part of $\partial\conv(\A)$, for which a basis $(v_2,\dots,v_n)$ is positive if $\det(\nu,v_2,\dots,v_n)>0$ for an outward vector $\nu$. Let $i$ be an index such that $a_i\notin\Gamma$, so that $a_i-a_{s_1}$ points inward. Subtracting the first column from the others, we get
\[
\det A^{\tau\cup\{i\}}=\det\big(a_{s_2}-a_{s_1},\dots,a_{s_n}-a_{s_1},a_i-a_{s_1}\big)=(-1)^{n-1}\det\big(a_i-a_{s_1},a_{s_2}-a_{s_1},\dots,a_{s_n}-a_{s_1}\big).
\]
Therefore the two orientations of $\tau$ differ by the sign $(-1)^{n-1}\cdot(-1)\cdot\sign\det A^{\tau\cup\{i\}}=(-1)^n\sign\det A^{\tau\cup\{i\}}$. Note that this sign does not depend on $i$, since all the points $a_i\notin\Gamma$ lie on the same side of $\Gamma$. Summing the local degrees $(-1)^n\sign\det A^{\tau\cup\{i_\tau\}}\sign\det C^\tau$ over the simplices $\tau$ such that $u\in\cone(C_s,\ s\in\tau)$, we obtain the formula.
\end{proof}

\begin{proof}[Proof of Theorem~\ref{thm:A}]
Assume (U$_A$) and (U$_C$). By Lemma~\ref{lem:N}, (N) holds, and by Proposition~\ref{prop:degree}(1), $\deg F$ is the degree of $\hat F/|\hat F|$ on $\partial P$. By (U$_A$), every proper face $G$ of $P$ is a simplex whose only points in $\A$ are its vertices. Hence the boundary complex $T$ of $P$ is a triangulation of $\partial P$ as in Lemma~\ref{lem:PL}, whose $(n-1)$-simplices are the facets of $P$. For $q$ in the relative interior of $G$, the weights $\hat\pi_i(q)$, $i\in G$, are nonnegative, with sum $1$ and barycenter $q$ (Proposition~\ref{prop:moment}(2)). Thus they are the barycentric coordinates of $q$ in the simplex $G$, and $\hat F=\Pi_T$ on $\partial P$. This map does not vanish by (N), and Lemma~\ref{lem:PL} gives the formula.
\end{proof}

\begin{proof}[Proof of Theorem~\ref{thm:Ap}]
By Lemma~\ref{lem:homotopy}, $\deg F$ is the degree of $\Pi_T/|\Pi_T|$ on $\partial P$, and Lemma~\ref{lem:PL} gives the formula.
\end{proof}

\begin{proposition}\label{prop:ends-newton}
Assume (H), and that neither $e_n$ nor $-e_n$ lies in a cone spanned by $n-1$ columns of $C$. Let $C'$ be the matrix formed by the first $n-1$ rows of $C$. Then, for any triangulation $T$ of $\partial P$ as in Theorem~\ref{thm:Ap},
\[
\deg F=\frac{(-1)^n}2\sum_{\substack{\tau\in T,\ \dim\tau=n-1\\ (C'_s)_{s\in\tau}\ \text{positively span}\ \R^{n-1}}}\sign\det C^\tau\cdot\sign\det A^{\tau\cup\{i_\tau\}} .
\]
By Cramer's rule, $(C'_s)_{s\in\tau}$ positively spans $\R^{n-1}$ if and only if the numbers $(-1)^r\det C'^{\tau\setminus\{s_r\}}$, $r=1,\dots,n$, are nonzero and have the same sign, where $\tau=\{s_1<\dots<s_n\}$.
\end{proposition}

The assumption holds in particular if all the maximal minors of $C'$ are nonzero, since $\det(C^S,e_n)=\pm\det C'^S$ for $|S|=n-1$. It holds after replacing $F$ by $MF$ for a generic $M\in GL_n(\R)$, since a cone spanned by $n-1$ columns of $C$ lies in a proper subspace of $\R^n$, and a generic line through the origin meets such a cone only at the origin. This does not change the positive zeros, and it multiplies $\deg F$ by $\sign\det M$. In other words, the vector $u$ of Theorem~\ref{thm:Ap} can be replaced by a generic change of basis. The same holds on the Gale side (Proposition~\ref{prop:ends-gale}).

\begin{proof}
Let $g=\Pi_T/|\Pi_T|\colon\partial P\to S^{n-1}$, which is defined by Lemma~\ref{lem:homotopy}, and whose degree is $\deg F$. By assumption, Lemma~\ref{lem:PL} applies both with $u=e_n$ and with $u=-e_n$: $\deg g$ is the signed count of the points of $g^{-1}(e_n)$, and also of the points of $g^{-1}(-e_n)$, where a point of an $(n-1)$-simplex $\tau$ is counted with the sign $c_\tau=(-1)^n\sign\det C^\tau\cdot\sign\det A^{\tau\cup\{i_\tau\}}$. Note that this sign does not depend on $u$. Averaging the two counts, we get
\[
\deg F=\frac12\sum_{q\in g^{-1}(\{e_n,-e_n\})}c_{\tau(q)},
\]
where $\tau(q)$ is the simplex containing $q$.

The points of $g^{-1}(\{e_n,-e_n\})$ are the points of $\partial P$ where the first $n-1$ components of $\Pi_T$ vanish. Let us find them in an $(n-1)$-simplex $\tau=\{s_1<\dots<s_n\}$ of $T$. A point $q=\sum_r\beta_ra_{s_r}$ of $\tau$, with $\beta_r\ge0$ and $\sum_r\beta_r=1$, is such a point if and only if $\sum_r\beta_rC'_{s_r}=0$. Assume that $q$ is such a point. Then $\Pi_T(q)=\sum_r\beta_rC_{s_r}$ is a nonzero multiple of $e_n$, by Lemma~\ref{lem:homotopy}. If some $\beta_r$ were zero, $e_n$ or $-e_n$ would lie in a cone spanned by $n-1$ columns of $C$. Hence all the $\beta_r$ are positive. For the same reason, $\det C^\tau\neq0$ by Carath\'eodory's theorem, so the $(n-1)\times n$ matrix $C'^\tau$ has rank $n-1$, and its kernel is a line. Therefore the vectors $C'_s$, $s\in\tau$, positively span $\R^{n-1}$, and $q$ is the only such point of $\tau$. Conversely, assume that the vectors $C'_s$, $s\in\tau$, positively span $\R^{n-1}$, and let $\beta$ be a positive vector of the kernel of $C'^\tau$, normalized so that $\sum_r\beta_r=1$. Then $q=\sum_r\beta_ra_{s_r}$ is a point of $\tau$ where the first $n-1$ components of $\Pi_T$ vanish. Thus $\tau$ contains such a point if and only if the vectors $C'_s$, $s\in\tau$, positively span $\R^{n-1}$, and in this case it contains exactly one, $q_\tau$, which lies in the relative interior of $\tau$. This gives the formula. Finally, by Cramer's rule, the kernel of $C'^\tau$ is spanned by the vector $\big((-1)^r\det C'^{\tau\setminus\{s_r\}}\big)_r$, which gives the last assertion.
\end{proof}

\begin{remark}[The curve $f_1=\dots=f_{n-1}=0$]\label{rem:curve-newton}
Proposition~\ref{prop:ends-newton} says that $\deg F$ is half the signed count of the ends of the curve defined by the first $n-1$ components of $F$, the points $q_\tau$ playing the role of these ends. Let us explain this. Assume (U$_A$) and (U$_C$), take for $T$ the boundary complex of $P$, so that $\hat F=\Pi_T$ on $\partial P$, and assume that the positive zeros of $F$ are nondegenerate. After replacing $F$ by $MF$ for a generic $M\in GL_n(\R)$, the set $Z=\{x\in\R^n_{>0}:\ f_1(x)=\dots=f_{n-1}(x)=0\}$ is a smooth curve; indeed $Z=F^{-1}(L)$ for the line $L=M^{-1}(\R e_n)$, whose direction is a regular value of $F/|F|$ for generic $M$ by Sard's theorem, and $F$ is a local diffeomorphism near its zeros. We orient $Z$ by the vector field $\theta$ defined by $\langle\theta,w\rangle=\det(Df_1;\dots;Df_{n-1};w)$. Then at a positive zero $x$ of $F$ we have $\det DF(x)=\langle\theta,\nabla f_n(x)\rangle$, so that the Jacobian sign of $F$ at $x$ is the sign of the derivative of $f_n$ along $Z$. Hence the Jacobian signs alternate along each component of $Z$. A compact component contributes $0$ to $\deg F$. A noncompact component contributes $(\sigma_+-\sigma_-)/2\in\{-1,0,1\}$, where $\sigma_+$ and $\sigma_-$ are the signs of $f_n$ near its outgoing and incoming ends. In particular it contributes $\pm1$ if and only if $f_n$ takes opposite signs near its two ends.

Since $\hat F$ does not vanish on $\partial P$, the ends of $Z$ accumulate, through the moment map, on the points of $\partial P$ where the first $n-1$ components of $\hat F=\Pi_T$ vanish, that is, on the points $q_\tau$. Near $q_\tau$, the sign of $f_n$ is the sign of the last coordinate of $\Pi_T(q_\tau)$, that is, it is $+$ if $q_\tau\in g^{-1}(e_n)$ and $-$ if $q_\tau\in g^{-1}(-e_n)$, where $g=\Pi_T/|\Pi_T|$ is as in the proof of Proposition~\ref{prop:ends-newton}. Using the implicit function theorem for the facial system $F_\tau$ near $q_\tau$, one sees that exactly one end of $Z$ converges to each point $q_\tau$, and that the contribution of this end, $\sigma_+/2$ for an outgoing end and $-\sigma_-/2$ for an incoming one, is $c_\tau/2$. In particular, if $f_n$ takes opposite signs at the two ends of a component, these ends converge to two points $q_\tau$ and $q_{\tau'}$, one in $g^{-1}(e_n)$ and the other in $g^{-1}(-e_n)$, and the contribution $\pm1$ of the component is $c_\tau=c_{\tau'}$. If $f_n$ has the same sign at both ends, then $c_\tau=-c_{\tau'}$, and the component contributes $0$. The analogous description on the Gale side, for the curve defined in $\Delta$ by the first $k-1$ equations of the Gale system, is given after Proposition~\ref{prop:ends-gale}.
\end{remark}

\begin{proof}[Proof of Theorem~\ref{thm:comb}]
If $C$ has rank less than $n$, then $\Pi_T$ takes its values in a hyperplane. Hence $\Pi_T/|\Pi_T|$ is not onto, and $\deg F=0$ by Lemma~\ref{lem:homotopy}. Assume now that $C$ has rank $n$, and for a real number $\eta>0$ put $u_\eta=C_{j_1}+\eta C_{j_2}+\dots+\eta^{n-1}C_{j_n}$. A proper subspace of $\R^n$ spanned by columns of $C$ contains $u_\eta$ for at most finitely many $\eta$, since otherwise it would contain the basis $C_{j_r}$. Therefore, for $\eta>0$ small, $u_\eta$ lies in no cone spanned by $n-1$ columns, and Theorem~\ref{thm:Ap} applies with $u=u_\eta$. If the columns $C_s$, $s\in\tau$, are linearly dependent, then $u_\eta\notin\cone(C_s,\ s\in\tau)$. Otherwise $u_\eta=\sum_{s\in\tau}\lambda_s(\eta)C_s$ with $\lambda_s(\eta)=\sum_r\eta^{r-1}\mu^\tau_{s,r}$. This polynomial in $\eta$ is not identically zero, since $\mu^\tau_{s,1},\dots,\mu^\tau_{s,n}$ are the coordinates on $C_s$ of a basis of $\R^n$. For $\eta$ small, its sign is the sign of its first nonzero coefficient. Thus $u_\eta\in\cone(C_s,\ s\in\tau)$ if and only if $\tau$ is selected, and Theorem~\ref{thm:Ap} gives the formula.

For the last assertion, we fix $T$. The faces of $P$ and the triangulations of a point configuration are determined by its oriented matroid \cite[Section~4.1]{DLRS}. Hence $T$ is a triangulation as in Theorem~\ref{thm:Ap} for every configuration with the same oriented matroid as $\A$. Moreover, the rank of $C$, the selected sets and the signs in the formula are determined by the signs of the maximal minors of $A$ and $C$.
\end{proof}

\begin{remark}\label{rem:comb}
A canonical choice of $T$ is the pulling triangulation \cite[Section~4.3]{DLRS}. If $v(G)$ denotes the smallest index of a vertex of a face $G$, its $(n-1)$-simplices are the sets $\{v(G_0),\dots,v(G_{n-1})\}$ for the flags of faces $G_0\subset\dots\subset G_{n-1}$ with $\dim G_k=k$ and $v(G_k)\notin G_{k-1}$. By Proposition~\ref{lem:vertex}, one may also replace $F$ by $F_V$ and use a triangulation with vertices among the vertices of $P$. When $P$ is simplicial, $T$ can be the boundary complex of $P$. Then Theorem~\ref{thm:comb} is the formula of Theorem~\ref{thm:A}, with the condition $u\in\cone(C_s,\ s\in\Gamma)$ replaced by the condition that $\Gamma$ is selected. In the same way, on the Gale side, Theorem~\ref{thm:Bp} holds with $e=-(B_{i_1}+\eta B_{i_2}+\dots+\eta^{k-1}B_{i_k})$, where $B_{i_1},\dots,B_{i_k}$ is the first basis formed by rows of $B$. The condition $e\in\cone(-B_s,\ s\in S)$ then becomes the same lexicographic condition on the coordinates of the $B_{i_r}$ in the basis $(B_s)_{s\in S}$. Note that no triangulation is needed there, since the vertices of $\overline\Delta$ are canonical. This vector $e$ is $-B^Th$ for the lexicographic heights $h_{i_r}=\eta^{r-1}$, $h_i=0$ otherwise.
\end{remark}

\begin{theorem}[Oriented matroid invariance]\label{cor:om}
Assume (H). Then $\deg F$ only depends on the oriented matroids of the columns of $A$ and of $C$ indexed by the vertices of $P$.
\end{theorem}

\begin{proof}
By Proposition~\ref{lem:vertex}, $F_V$ satisfies (H) and $\deg F=\deg F_V$. By Theorem~\ref{thm:comb} applied to $F_V$, $\deg F_V$ only depends on the oriented matroids of the columns of $A$ and of $C$ indexed by $V$.
\end{proof}

\subsection{One or two variables}\label{sec:onetwo}
For $n=1$, we write $a_1<\dots<a_N$. The facets of $\conv(\A)$ are $\{1\}$ and $\{N\}$, with $\varepsilon_{\{1\}}=1$ and $\varepsilon_{\{N\}}=-1$. Theorem~\ref{thm:A} with $u=1$ gives $\deg F=[c_N>0]-[c_1>0]$. In other words, the signed number of positive roots of $\sum_ic_ix^{a_i}$ is $0$ if $c_1c_N>0$, and $\pm1$ otherwise.

\begin{corollary}\label{cor:n2}
Let $n=2$, assume (H), and let $j_1,\dots,j_m$ be the indices of the points of $\A$ on the boundary of the Newton polygon $P$, in counterclockwise order. Then the closed polygon $C_{j_1}C_{j_2}\cdots C_{j_m}$ does not pass through the origin, and $\deg F$ is its winding number around the origin. The same holds for the closed polygon of the columns indexed by the vertices of $P$ only; under (U$_A$), these are all the points of $\A$ on $\partial P$.
\end{corollary}

\begin{proof}
Put $j_{m+1}=j_1$. By (H), for every edge $G$ of $P$ the columns $C_i$, $i\in G$, lie in an open half-plane. Hence no segment $[C_{j_l},C_{j_{l+1}}]$ contains the origin. Let $T$ be the triangulation of $\partial P$ whose $1$-simplices are the segments $\tau_l=\{j_l,j_{l+1}\}$; every edge of $P$ is a union of them. If $j_l<j_{l+1}$, a point $a_i$ off the edge containing $\tau_l$ lies to the left of the oriented segment from $a_{j_l}$ to $a_{j_{l+1}}$, so $\sign\det A^{\tau_l\cup\{i\}}=+1$. If $j_l>j_{l+1}$, the columns of $\tau_l$ are taken in the other order, and this sign is $-1$. In both cases, $\sign\det C^{\tau_l}\cdot\sign\det A^{\tau_l\cup\{i\}}=\sign\det(C_{j_l},C_{j_{l+1}})$. Let $u$ be as in Theorem~\ref{thm:A}, that is, $u$ does not lie on any ray $\R_{\ge0}C_i$. If $\det(C_{j_l},C_{j_{l+1}})=0$, the two vectors point in the same direction. Then $u\notin\cone(C_{j_l},C_{j_{l+1}})$, and the segment $[C_{j_l},C_{j_{l+1}}]$ does not meet the ray $\R_{>0}u$. Otherwise, $u\in\cone(C_{j_l},C_{j_{l+1}})$ if and only if this segment crosses the ray $\R_{>0}u$. Therefore, by Theorem~\ref{thm:Ap} with $(-1)^n=1$, $\deg F$ is the number of segments $[C_{j_l},C_{j_{l+1}}]$ which cross the ray $\R_{>0}u$, counted with the sign of $\det(C_{j_l},C_{j_{l+1}})$. This is the winding number. The same argument, with the triangulation of $\partial P$ by the edges of $P$, gives the assertion for the vertices.
\end{proof}

\begin{example}[Example~\ref{ex:run}, continued: Theorems~\ref{thm:A} and~\ref{thm:comb}]\label{ex:run-A}
The facets of $P$ are the five edges. For each of them we list $\det C^\Gamma$ and $\varepsilon_\Gamma$:
\[
\begin{array}{c|ccccc}
\Gamma&\{1,2\}&\{2,3\}&\{3,4\}&\{4,5\}&\{1,5\}\\ \hline
\det C^\Gamma&6&7&6&7&-6\\
\varepsilon_\Gamma&+1&+1&+1&+1&-1
\end{array}
\]
Thus every term of Theorem~\ref{thm:A} is $+1$, and $\deg F$ is the number of edges $\Gamma$ such that $u\in\cone(C_s,\ s\in\Gamma)$. This number is $2$ for every admissible $u$. For instance, for $u=(0,1)$ the selected edges are $\{1,2\}$ and $\{4,5\}$, for $u=(-1,-1)$ they are $\{2,3\}$ and $\{1,5\}$, and for $u=(1,1)$ they are $\{1,2\}$ and $\{3,4\}$. In Theorem~\ref{thm:comb}, the first basis formed by columns of $C$ is $C_1,C_2$, and $u_\eta=C_1+\eta C_2$ selects $\{1,2\}$ and $\{3,4\}$. In Proposition~\ref{prop:ends-newton}, $C'=(3,-3,1,1,-3)$ is the first row of $C$, whose entries are nonzero. An edge contributes if and only if its two entries have opposite signs. These edges are $\{1,2\}$, $\{2,3\}$, $\{4,5\}$ and $\{1,5\}$, and we get $\deg F=\frac12\cdot4=2$. Finally, the closed polygon $C_1C_2C_3C_4C_5$ turns twice around the origin, as predicted by Corollary~\ref{cor:n2}.
\end{example}

\subsection{The Viro formula by a piecewise linear argument}
Let $h\in\R^N$ be generic, and put $\Pi_h=\Pi_{T_h}\colon P\to\R^n$. The following proposition gives the degree formula of Theorem~\ref{thm:viro} without using the zeros of the Viro maps. The idea is that the degree of $F$ is the degree of $\Pi_h$, which can be read on $\partial P$ through Lemma~\ref{lem:homotopy}, and in the interior of $P$ as a sum over the simplices of $T_h$.

\begin{proposition}\label{prop:viro-pl}
Assume (U$_C$) and (H), and let $h\in\R^N$ be generic. Then
\[
\deg F=\deg(\Pi_h,P^\circ,0)=\sum_{\substack{\sigma\in T_h\\ \sigma\ \text{positively decorated}}}s(\sigma).
\]
\end{proposition}

\begin{proof}
Every face of $P$ is a union of simplices of $T_h$. Hence, by Lemma~\ref{lem:homotopy}, $\Pi_h$ does not vanish on $\partial P$, and $\deg F$ is the degree of $\Pi_h/|\Pi_h|$ on $\partial P$. As in the proof of Proposition~\ref{prop:degree}, this is $\deg(\Pi_h,P^\circ,0)$ by the Kronecker--Poincar\'e--Hopf index theorem.

On a simplex of $T_h$ of dimension less than $n$, $\Pi_h$ is a nonnegative nonzero combination of at most $n$ columns of $C$. These columns are linearly independent by (U$_C$), so $\Pi_h$ does not vanish there. On the simplex of $\sigma=\{l_0<\dots<l_n\}\in T_h$, $\Pi_h$ is affine with image $\conv(C_l,\ l\in\sigma)$. Therefore it vanishes at an interior point if and only if $\sigma$ is positively decorated, and then at exactly one point. The local degree at this point is the sign of the determinant of the linear part of $\Pi_h$, which sends $a_{l_j}-a_{l_0}$ to $C_{l_j}-C_{l_0}$:
\[
\det\big(C_{l_j}-C_{l_0}\big)_{j=1}^n\cdot\det\big(a_{l_j}-a_{l_0}\big)_{j=1}^n=\det\widehat C^\sigma\cdot\det A^\sigma ,
\]
as one sees by subtracting the first column from the others in both matrices. Thus the local degree is $s(\sigma)$, and the formula follows.
\end{proof}

In particular, the signed count of the positively decorated simplices of $T_h$ does not depend on $h$. Indeed, it is the degree of a piecewise linear map whose restriction to $\partial P$ is homotopic to $\hat F$. Under (U$_A$), every proper face of $P$ is a simplex whose only points in $\A$ are its vertices, so $\Pi_h=\hat F$ on $\partial P$, as in the proof of Theorem~\ref{thm:A}.

\section{Theorem~\ref{thm:A} as a Viro formula}\label{sec:coning}

In this section we show that the sum of Theorem~\ref{thm:A} is itself a Viro sum. More precisely, each of its terms is the sign $s(\sigma)$ of a simplex of a coning triangulation, obtained by adding to $\A$ an interior point decorated by $-u$. This gives a second proof of Theorem~\ref{thm:A}, and it explains the sign $(-1)^n$.

Assume (U$_A$) and (U$_C$). Let $u\in\R^n$ be a vector which does not lie in any linear hyperplane spanned by $n-1$ columns of $C$. In particular, $u$ does not lie in any cone spanned by $n-1$ columns. We choose a point $c$ in the interior of $P$ which does not lie on any affine hyperplane spanned by $n$ points of $\A$. We put $a_{N+1}=c$, $\A'=\A\cup\{c\}$ and $C_{N+1}=-u$, we let $C'=(C_1,\dots,C_N,C_{N+1})$ and $A'$ be the corresponding matrices, and we consider
\[
F'(x)=\sum_{i=1}^{N+1}C_i\,x^{a_i}=F(x)-u\,x^{c}.
\]

\begin{lemma}\label{lem:coning1}
The pair $(\A',C')$ satisfies (U$_A$) and (U$_C$), its Newton polytope is $P$, and $\deg F'=\deg F$.
\end{lemma}

\begin{proof}
A maximal minor of $A'$ either is a maximal minor of $A$, or involves $c$ and $n$ points of $\A$. In both cases it is nonzero, by (U$_A$) and by the choice of $c$. A maximal minor of $C'$ either is a maximal minor of $C$, or is $\pm\det(C_{t_1},\dots,C_{t_{n-1}},u)$, which is nonzero by the choice of $u$. Since $c$ lies in the interior of $P$, we have $\conv(\A')=P$, and $\A'$ and $\A$ have the same points on each proper face of $P$. Hence, for $\theta\in[0,1]$, the maps $F-\theta u\,x^c$ have the same truncations to the proper faces of $P$ as $F$, and they satisfy (N) by Lemma~\ref{lem:N}(2). By Proposition~\ref{prop:degree}(3), applied to the configuration $\A'$ and the coefficient matrices $(C_1,\dots,C_N,-\theta u)$, $\theta\in[0,1]$, we get $\deg F'=\deg F$. Note that for $\theta=0$ the map is $F$, and by Definition~\ref{def:deg} its degree does not depend on whether we view it with the exponent set $\A$ or $\A'$.
\end{proof}

\begin{lemma}\label{lem:coning2}
Let $h\in\R^{N+1}$ with $h_i=0$ for $i\le N$ and $h_{N+1}=-1$. Then $h$ is generic for $\A'$, and the simplices of $T_h$ are the sets $\sigma_\Gamma=\Gamma\cup\{N+1\}$, where $\Gamma$ runs over the facets of $P$.
\end{lemma}

\begin{proof}
Let $\Gamma$ be a facet of $P$, and let $\ell_\Gamma$ be the affine function which vanishes on $\Gamma$ and is positive on the interior of $P$. By (U$_A$), $\ell_\Gamma(a_i)>0$ for $i\notin\Gamma$. The affine function $\lambda=-\ell_\Gamma/\ell_\Gamma(c)$ satisfies $\lambda(a_s)=0=h_s$ for $s\in\Gamma$, $\lambda(c)=-1=h_{N+1}$, and $\lambda(a_i)<0=h_i$ for $i\le N$, $i\notin\Gamma$. Thus $\{(a_i,h_i)\}_{i\in\sigma_\Gamma}$ spans a lower facet of the lifted configuration which contains no other lifted point, that is, $\sigma_\Gamma\in T_h$. Since $P$ is the cone from its interior point $c$ over its boundary, the simplices $\conv(\Gamma\cup\{c\})$, $\Gamma$ a facet of $P$, cover $P$. Hence they are all the simplices of $T_h$, and every lower facet contains exactly $n+1$ lifted points.
\end{proof}

\begin{lemma}\label{lem:coning3}
Let $\Gamma$ be a facet of $P$, and $\sigma_\Gamma=\Gamma\cup\{N+1\}$ the corresponding simplex of $T_h$ (Lemma~\ref{lem:coning2}). Then $\sigma_\Gamma$ is positively decorated by $C'$ if and only if $u\in\cone(C_s,\ s\in\Gamma)$, and in this case
\[
s(\sigma_\Gamma)=(-1)^n\,\sign\det C^\Gamma\cdot\varepsilon_\Gamma .
\]
\end{lemma}

\begin{proof}
The matrix $C'^{\sigma_\Gamma}$ has the columns $C_s$, $s\in\Gamma$, in increasing order, followed by $-u$. It is positively spanning if and only if its kernel is spanned by a positive vector, that is, if and only if $\sum_{s\in\Gamma}\lambda_sC_s=\mu u$ with $\lambda_s>0$ and $\mu>0$. In other words, $\sigma_\Gamma$ is positively decorated if and only if $u$ lies in the interior of $\cone(C_s,\ s\in\Gamma)$. Since $u$ does not lie in a cone spanned by $n-1$ columns, this is equivalent to $u\in\cone(C_s,\ s\in\Gamma)$.

Assume now that $\sigma_\Gamma$ is positively decorated. The matrix $A'^{\sigma_\Gamma}$ has the columns $(1,a_s)$, $s\in\Gamma$, in increasing order, followed by $(1,c)$. The function $x\mapsto\det A^{\Gamma\cup\{x\}}$, where the last column is $(1,x)$, is affine and vanishes on the affine hull of $\Gamma$. Moreover, $c$ and $a_{i_\Gamma}$ lie strictly on the same side of this hull. Hence $\sign\det A'^{\sigma_\Gamma}=\sign\det A^{\Gamma\cup\{i_\Gamma\}}=\varepsilon_\Gamma$. Expanding $\det\widehat C'^{\sigma_\Gamma}$ along its first row, we get
\[
\det\widehat C'^{\sigma_\Gamma}=\sum_{j=1}^{n+1}(-1)^{1+j}\operatorname{minor}\big(C'^{\sigma_\Gamma},j\big).
\]
Since $C'^{\sigma_\Gamma}$ is positively spanning, the numbers $(-1)^j\operatorname{minor}(C'^{\sigma_\Gamma},j)$ are nonzero and have the same sign. Thus all the terms of this sum have the same sign, and $\sign\det\widehat C'^{\sigma_\Gamma}$ is the sign of the last term, $(-1)^{n+2}\operatorname{minor}(C'^{\sigma_\Gamma},n+1)=(-1)^n\det C^\Gamma$. Multiplying, we obtain $s(\sigma_\Gamma)=\varepsilon_\Gamma\cdot(-1)^n\sign\det C^\Gamma$.
\end{proof}

\begin{proof}[Second proof of Theorem~\ref{thm:A}]
Assume first that $u$ does not lie in any linear hyperplane spanned by $n-1$ columns of $C$. By Lemma~\ref{lem:coning1}, $\deg F=\deg F'$, and $(\A',C')$ satisfies (U$_A$) and (U$_C$). By Theorem~\ref{thm:viro}(2) applied to $F'$ with the heights of Lemma~\ref{lem:coning2}, and by Lemma~\ref{lem:coning3}, we get
\[
\deg F=\deg F'=\sum_{\substack{\Gamma\ \text{facet of}\ P\\ \sigma_\Gamma\ \text{positively decorated}}}s(\sigma_\Gamma)=(-1)^n\sum_{\substack{\Gamma\ \text{facet of}\ P\\ u\in\cone(C_s,\ s\in\Gamma)}}\sign\det C^\Gamma\cdot\varepsilon_\Gamma ,
\]
which is the formula of Theorem~\ref{thm:A}. In general, the right-hand side is locally constant on the set of vectors $u$ which do not lie in any cone spanned by $n-1$ columns of $C$. Indeed, the set of facets $\Gamma$ such that $u\in\cone(C_s,\ s\in\Gamma)$ can only change when $u$ crosses the boundary of such a cone, and this boundary is covered by cones spanned by $n-1$ columns. Each connected component of this open set contains vectors which avoid the hyperplanes spanned by $n-1$ columns. Therefore the formula holds for every $u$ as in Theorem~\ref{thm:A}.
\end{proof}

To summarize, the choice of the direction $u$ in Theorem~\ref{thm:A} is the choice of the decoration $-u$ of an auxiliary interior point. The facets $\Gamma$ which contribute are the positively decorated simplices of the coning triangulation from this point, and the sign $(-1)^n$ comes from the position of the column $-u$ in $\widehat C'^{\sigma_\Gamma}$. For instance, for $n=1$, with $P=[a_1,a_N]$ and $u=1$, the segment $[a_1,c]$ is positively decorated if and only if $c_1>0$, with $s=-1$, and the segment $[c,a_N]$ if and only if $c_N>0$, with $s=+1$. This gives again $\deg F=[c_N>0]-[c_1>0]$.

\section{Gale duality and the second formula}\label{sec:gale}

\subsection{Gale dual matrices}\label{sec:gale-matrices}
Let $B$ be an $N\times k$ matrix whose columns form a basis of $\ker A$, and let $D$ be an $N\times(k+1)$ matrix whose columns form a basis of $\ker C$. We say that $B$ and $D$ are \emph{Gale dual matrices} of $A$ and $C$. Since the first row of $A$ is $(1,\dots,1)$, the rows of $B$ sum to zero: $\sum_iB_i=0$. We denote by $B_i\in\R^k$ and $D_i\in\R^{k+1}$ the $i$-th rows, and for $S\subset[N]$ we denote by $B_S$ and $D_S$ the submatrices with rows in $S$, in increasing order. Under (U$_C$), $D_N\neq0$, and we normalize $D$ so that $D_N=(0,\dots,0,1)$.

We will use the following classical relation between complementary minors \cite[Section~3.4]{BLSWZ}. If $M$ is an $r\times N$ matrix of rank $r$ and the columns of the $N\times(N-r)$ matrix $G$ form a basis of $\ker M$, then there is a constant $c\neq0$ such that
\begin{equation}\label{eq:complementary}
\det M^\tau=c\,(-1)^{\sum_{j\in\tau}j}\,\det G_{[N]\setminus\tau}\qquad\text{for all }\tau\subset[N],\ |\tau|=r .
\end{equation}
In particular, (U$_A$) is equivalent to the condition that any $k$ rows of $B$ are linearly independent, and (U$_C$) to the condition that any $k+1$ rows of $D$ are linearly independent.

We put $P_i(v)=\langle D_i,v\rangle$ for $v\in\R^{k+1}$ and $p_i(y)=P_i(y,1)$ for $y\in\R^k$, so that $p_N=1$, and we consider
\[
\Delta=\{y\in\R^k:\ p_i(y)>0,\ i\in[N]\},\qquad K=\{v\in\R^{k+1}:\ P_i(v)\ge0,\ i\in[N]\}.
\]
The vectors of $\ker C$ are the vectors $Dv$. Hence $\Delta\neq\emptyset$ if and only if $\ker C$ contains a positive vector. When $\Delta\neq\emptyset$, $K$ is a pointed cone with nonempty interior, and $\overline\Delta$ is identified with the image of $K\setminus\{0\}$ in the sphere $S^k$. For $S\subset[N]$, let $K_S=K\cap\{P_s=0,\ s\in S\}$. We say that $S$ is a \emph{face} of $\Delta$ if $K_S\neq\{0\}$ and $P_i>0$ on the relative interior of $K_S$ for $i\notin S$. We denote by $V(\Delta)$ the set of the faces of size $k$. They correspond to the vertices of the closure $\overline\Delta$ of $\Delta$ in the projective space, that is, to the extreme rays of $K$, including the vertices at infinity when $\Delta$ is unbounded.

\begin{proposition}\label{prop:vertices}
Assume (U$_C$) and $\Delta\neq\emptyset$.
\begin{enumerate}
\item Every face $S$ of $\Delta$ has $|S|\le k$, and every subset of a face is a face. The sets $S\in V(\Delta)$ correspond to the extreme rays of $K$.
\item A set $S$ with $|S|=k$ belongs to $V(\Delta)$ if and only if the numbers $\det(D_S;D_i)$, $i\notin S$, all have the same sign.
\end{enumerate}
\end{proposition}

\begin{proof}
By (U$_C$), any $k+1$ of the linear forms $P_i$ are linearly independent. Hence a nonzero face of $K$ lies on at most $k$ of the hyperplanes $\{P_i=0\}$. For $|S|=k$, the line $\{P_s=0,\ s\in S\}$ is spanned by the cofactor vector $v_S$ of $D_S$, which satisfies $\langle D_i,v_S\rangle=\det(D_S;D_i)$. This line meets $K$ in an extreme ray if and only if $\pm v_S\in K$, that is, if and only if the nonzero numbers $\det(D_S;D_i)$, $i\notin S$, have the same sign. Finally, every nonzero face of $K$ is spanned by its extreme rays, and the rest follows from the simplicity of the arrangement.
\end{proof}

Let $\lambda=\sum_iD_i$, so that $\lambda(v)=\sum_iP_i(v)>0$ on $K\setminus\{0\}$, and let $Q=\{v\in K:\ \lambda(v)=1\}$. This is a compact simple polytope of dimension $k$. Its vertices are the points $Q\cap K_S$, $S\in V(\Delta)$, and its facets lie on the hyperplanes $\{P_i=0\}$. In particular, $Q$ has at most $N$ facets.

\subsection{Gale translation of the Viro formula}

\begin{lemma}\label{lem:translation}
Assume (U$_C$). Let $h\in\R^N$, put $e=-B^Th$, and let $\sigma\subset[N]$ with $|\sigma|=n+1$ and $S=[N]\setminus\sigma$, so that $|S|=k$.
\begin{enumerate}
\item If $\det B_S\neq0$, then $\sigma\in T_h$ if and only if $e$ lies in the interior of the simplicial cone $\cone(-B_s,\ s\in S)$. If $e$ does not lie in a cone spanned by $k-1$ of the vectors $-B_i$, then $h$ is generic, and the simplices of $T_h$ are the sets $\sigma$ such that $e\in\cone(-B_s,\ s\in S)$; for these, $\det B_S\neq0$. Under (U$_A$), $h$ is generic if and only if $e$ does not lie in a cone spanned by $k-1$ of the vectors $-B_i$.
\item $\sigma$ is positively decorated if and only if $\Delta\neq\emptyset$ and $S\in V(\Delta)$.
\item There is a sign $\epsilon\in\{\pm1\}$, depending only on the choices of $B$ and $D$, such that $\omega(S)=\epsilon\,s(\sigma)$ for every positively decorated $\sigma\in T_h$.
\end{enumerate}
\end{lemma}

\begin{proof}
1. By~\eqref{eq:complementary}, $\det B_S\neq0$ if and only if $\det A^\sigma\neq0$, that is, if and only if the points $a_l$, $l\in\sigma$, are affinely independent. Let $\sigma'\subset[N]$ be any subset, with complement $S'$. The points $(a_i,h_i)$, $i\in\sigma'$, span a lower face of the lifted configuration which contains no other lifted point if and only if $h=A^Tc+r$ for some $c\in\R^{n+1}$ and $r\in\R^N$ with $r_{\sigma'}=0$ and $r_{S'}>0$. In this case $B^Th=B^Tr=\sum_{s\in S'}r_sB_s$, so $e=\sum_{s\in S'}r_s(-B_s)$. Conversely, assume that $e=\sum_{s\in S'}r_s(-B_s)$ with $r_s>0$, and extend $r$ by $0$ on $\sigma'$. Then $B^T(h-r)=0$, so $h-r$ lies in the row space of $A$, which is the orthogonal complement of $\ker A$. Taking $\sigma'=\sigma$ with $\det A^\sigma\neq0$, we get the first assertion.

Assume now that $e$ does not lie in a cone spanned by $k-1$ of the vectors $-B_i$. Then a lower facet cannot contain more than $n+1$ lifted points, since its complement $S'$ would have fewer than $k$ elements. Hence the lower facets are simplices, and $h$ is generic. Moreover, if $\det B_S=0$, then by Carath\'eodory's theorem $\cone(-B_s,\ s\in S)$ is covered by cones spanned by $k-1$ of the vectors $-B_s$, so $e$ does not lie in it. If $\det B_S\neq0$ and $e$ lies in this cone, then it lies in its interior. Finally, under (U$_A$) all the cones $\cone(-B_s,\ s\in S)$ are simplicial. The converse then follows, since $e$ lies in a cone spanned by fewer than $k$ vectors $-B_i$ if and only if it lies in the boundary of some $\cone(-B_s,\ s\in S)$.

2. By (U$_C$), the kernel of $C^\sigma$ is a line. Thus $\sigma$ is positively decorated if and only if $\ker C$ contains a vector $m=Dv$ with $m_\sigma>0$ and $m_S=0$, that is, such that $P_s(v)=0$ for $s\in S$ and $P_i(v)>0$ for $i\in\sigma$. Moving $v$ slightly in a direction where the $k$ independent forms $P_s$, $s\in S$, are positive, we get an interior point of $K$, so $\Delta\neq\emptyset$. Moreover, $v$ spans an extreme ray of $K$ with $P_i>0$ off $S$, that is, $S\in V(\Delta)$. The converse is Proposition~\ref{prop:vertices}.

3. Let $i\in\sigma$, and let $p$ be the position of $i$ in $\sigma$. Expanding $\det\widehat C^\sigma$ along its first row gives $\sum_{j}(-1)^{1+j}\det C^{\sigma\setminus\{l_j\}}$. Up to a common sign, the terms of this sum are the coordinates of a vector spanning $\ker C^\sigma$, so they have the same sign when $\sigma$ is positively decorated. Hence $\sign\det\widehat C^\sigma=(-1)^{1+p}\sign\det C^{\sigma\setminus\{i\}}$. By~\eqref{eq:complementary} applied to $(A,B)$ with $\tau=\sigma$, and to $(C,D)$ with $\tau=\sigma\setminus\{i\}$, we have
\[
\sign\det B_S=\pm(-1)^{\sum_{j\in\sigma}j}\sign\det A^\sigma,\qquad \sign\det D_{S\cup\{i\}}=\pm(-1)^{\sum_{j\in\sigma}j-i}\sign\det C^{\sigma\setminus\{i\}},
\]
with signs $\pm$ which only depend on $B$ and $D$. Moreover, $\det(D_S;D_i)=(-1)^{\#\{s\in S:\ s>i\}}\det D_{S\cup\{i\}}$. Multiplying, we get
\[
\omega(S)=\pm(-1)^{i+\#\{s\in S:\ s>i\}+1+p}\,s(\sigma).
\]
Since $i=p+\#\{s\in S:\ s<i\}$, the exponent is congruent to $k+1$ modulo $2$. Therefore $\omega(S)=\pm(-1)^{k+1}s(\sigma)$, with a sign which does not depend on $\sigma$.
\end{proof}

\begin{proof}[Proof of Theorems~\ref{thm:B} and~\ref{thm:Bp}]
Under (U$_A$) and (U$_C$), (H) holds by Lemma~\ref{lem:N}. Hence Theorem~\ref{thm:B} is a special case of Theorem~\ref{thm:Bp}, and we may assume (U$_C$) and (H). Since $B^T$ is onto, the generic vector $e$ is of the form $-B^Th$, and $h$ is then generic by Lemma~\ref{lem:translation}(1). If $\Delta=\emptyset$, then no simplex is positively decorated by Lemma~\ref{lem:translation}(2), and $\deg F=0$ by the Viro formula (Proposition~\ref{prop:viro-pl}). Otherwise, by Proposition~\ref{prop:viro-pl} and Lemma~\ref{lem:translation}, we get
\[
\deg F=\sum_{\substack{\sigma\in T_h\\ \text{pos.\ decorated}}}s(\sigma)=\epsilon\sum_{\substack{S\in V(\Delta)\\ e\in\cone(-B_s,\ s\in S)}}\omega(S)=\epsilon\deg(\Delta).\qedhere
\]
\end{proof}

Thus Theorem~\ref{thm:B} is the Viro formula written on the Gale side. The simplices of $T_h$ are the complements of the sets $S$ such that $e\in\cone(-B_s,\ s\in S)$, and the positively decorated ones correspond to the vertices of $\overline\Delta$. In particular, $\deg(\Delta)$ does not depend on $e$. Moreover, for $t$ small, the number of positive zeros of $F^{h,t}$ is $\#\{S\in V(\Delta):\ e\in\cone(-B_s,\ s\in S)\}\le\#V(\Delta)\le M(k,N)$, by the Upper Bound Theorem \cite{McM} applied to the simple $k$-polytope $Q$, which has at most $N$ facets.

\begin{proposition}\label{prop:ends-gale}
Assume (U$_C$), (H) and $\Delta\neq\emptyset$, and that neither $e_k$ nor $-e_k$ lies in a cone spanned by $k-1$ of the vectors $-B_i$. Let $B'$ be the matrix formed by the first $k-1$ columns of $B$. Then, with the global sign of Theorem~\ref{thm:B},
\[
\deg F=\pm\frac12\sum_{\substack{S\in V(\Delta)\\ (B'_s)_{s\in S}\ \text{positively span}\ \R^{k-1}}}\omega(S).
\]
\end{proposition}

The assumption holds in particular if all the maximal minors of $B'$ are nonzero. It holds after replacing $B$ by $BM$ for a generic $M\in GL_k(\R)$, that is, after a generic change of basis of $\ker A$. This does not change $\Delta$, and it only changes the global sign. As on the Newton side (Remark~\ref{rem:curve-newton}), the vertices counted in Proposition~\ref{prop:ends-gale} correspond to the ends of a curve. Let $\mathcal C$ be the curve defined in $\Delta$ by the first $k-1$ equations $\psi_1=\dots=\psi_{k-1}=0$ of the Gale system (Remark~\ref{rem:gale-system}). For generic data, the ends of $\mathcal C$ converge to the vertices $S\in V(\Delta)$ such that the vectors $B'_s$, $s\in S$, positively span $\R^{k-1}$, and each of these vertices is reached by exactly one end \cite{PartI}. The positive zeros of $F$ are the zeros of $\psi_k$ on $\mathcal C$, and Proposition~\ref{prop:ends-gale} says that $\deg F$ is half the signed count of the ends of $\mathcal C$. In particular, $|\deg F|$ is at most the number of noncompact components of $\mathcal C$, and the bound of Theorem~\ref{thm:descartes}(2) is the bound for this number obtained in \cite{PartI}.

\begin{proof}
By assumption, Theorem~\ref{thm:Bp} applies both with $e=e_k$ and with $e=-e_k$. By its proof, the global sign $\epsilon$ is the same in both cases. A vertex $S\in V(\Delta)$ contributes $\omega(S)$ to the formula for $e=\pm e_k$ if and only if $\pm e_k\in\cone(-B_s,\ s\in S)$. Averaging the two formulas, we get
\[
\deg F=\frac\epsilon2\sum_S\omega(S),
\]
where the sum runs over the vertices $S$ such that $\cone(-B_s,\ s\in S)$ meets the line $\R e_k$ outside the origin. Indeed, such a vertex contributes to exactly one of the two formulas: if $e_k$ or $-e_k$ lies in $\cone(-B_s,\ s\in S)$, then $\det B_S\neq0$ by Carath\'eodory's theorem, so this cone is simplicial and cannot contain both $e_k$ and $-e_k$.

It remains to characterize these vertices. The cone $\cone(-B_s,\ s\in S)$ meets $\R e_k$ outside the origin if and only if there is $\beta\in\R^S_{\ge0}$ such that $\sum_s\beta_sB'_s=0$ and $\sum_s\beta_sB_s\neq0$. The second condition is automatic for $\beta\neq0$. Indeed, if $\sum_s\beta_sB_s=0$ with $\beta\ge0$, $\beta\neq0$, then the support $S'\subset S$ of $\beta$ is a coface, and $K_{S'}\supset K_S\neq\{0\}$, which contradicts (H) by Lemma~\ref{lem:gale-N}(2). Now, assume that $\sum_s\beta_sB'_s=0$ with $\beta\ge0$, $\beta\neq0$. Then $\sum_s\beta_s(-B_s)$ is a nonzero multiple of $e_k$. If some $\beta_s$ were zero, $e_k$ or $-e_k$ would lie in a cone spanned by $k-1$ of the vectors $-B_i$. Hence all the $\beta_s$ are positive. For the same reason $\det B_S\neq0$, so the $k$ vectors $B'_s$, $s\in S$, span $\R^{k-1}$, and therefore they positively span $\R^{k-1}$. Conversely, if they positively span $\R^{k-1}$, they satisfy a linear relation with positive coefficients. This gives the formula.
\end{proof}

\begin{corollary}[A parity criterion for the existence of positive zeros]\label{cor:odd}
Assume (U$_A$) and (U$_C$). Let $h\in\R^N$ be generic, $u\in\R^n$ as in Theorem~\ref{thm:A} and $e\in\R^k$ as in Theorem~\ref{thm:B}. Then
\begin{align*}
\deg F&\equiv\#\{\sigma\in T_h\ \text{positively decorated}\}\\
&\equiv\#\{\Gamma\ \text{facet of}\ \conv(\A):\ u\in\cone(C_s,\ s\in\Gamma)\}\\
&\equiv\#\{S\in V(\Delta):\ e\in\cone(-B_s,\ s\in S)\}\pmod2 .
\end{align*}
(If $\Delta=\emptyset$, the last set is empty.)
In particular, if one of these numbers is odd for some choice of $h$, $u$ or $e$, then all of them are odd for every choice, $F$ has at least one positive zero, and if the positive zeros are nondegenerate, their number is odd.
\end{corollary}

\begin{proof}
Each of the three formulas of Theorems~\ref{thm:viro}, \ref{thm:A} and~\ref{thm:B} writes $\pm\deg F$ as a sum of signs $\pm1$, one for each element of the corresponding set. A sum of $m$ signs is congruent to $m$ modulo $2$. If $\deg F$ is odd, then it is nonzero, and Proposition~\ref{prop:degree}(2) and Corollary~\ref{cor:parity} apply.
\end{proof}

\begin{proof}[Proof of Theorem~\ref{thm:F}]
By Theorems~\ref{thm:Ap}, \ref{thm:viro}(2) and~\ref{thm:Bp}, each of the three sums is $\pm\deg F$. A sum of $m\ge1$ signs which are all equal is $\pm m\neq0$, and a sum of an odd number of signs is odd, hence nonzero. Thus $\deg F\neq0$, and $F$ has a positive zero by Proposition~\ref{prop:degree}(2). If the terms are all equal, then $|\deg F|$ is their number. If moreover the positive zeros are nondegenerate, there are at least $|\deg F|$ of them by Proposition~\ref{prop:degree}(1).
\end{proof}

\begin{example}[Example~\ref{ex:run}, continued: existence]\label{ex:deg2}
For the running example, each of the three sums of Theorem~\ref{thm:F} has two terms, both equal to $+1$, for every choice of $T$ and $u$, of $h$, or of $e$ (Examples~\ref{ex:run}, \ref{ex:run-A} and~\ref{ex:run-gale}). Hence $\deg F=2$, and $F$ has at least two positive zeros. Note that, since the degree is even, the parity criterion of Corollary~\ref{cor:odd} gives nothing here.
\end{example}

Observe that the criteria of Theorem~\ref{thm:F} only involve the conditions which select the terms and the signs of the selected terms. These conditions are whether a simplex is positively decorated, whether a simplex of $T$ carries $u$ in the cone of its columns, or whether a vertex of $\overline\Delta$ carries $e$ in its cone. In terms of Viro's construction, Corollary~\ref{cor:odd} says the following: if some regular triangulation of $\A$ has an odd number of positively decorated simplices, then not only the Viro systems $F^{h,t}$ for $t$ small, as in \cite[Theorem~3.4]{BSS}, but also the map $F$ itself has a positive zero. This criterion is of the same nature as the sign conditions of Bihan, Dickenstein and Giaroli \cite{BDG} for the existence of a positive solution.

\begin{remark}[Complex zeros on the Gale side]\label{rem:complex-gale}
Assume that the exponents are integers and that the vectors $(1,a_i)$ generate $\mathbb Z^{n+1}$. We choose for $B$ an integral matrix whose columns form a basis of the lattice $\ker A\cap\mathbb Z^N$. The row lattice of $A$ and $\ker A\cap\mathbb Z^N$ are then primitive and orthogonal, and the integral version of~\eqref{eq:complementary} holds with $c=\pm1$: $|\det B_S|=|\det A^{[N]\setminus S}|$ for $|S|=k$. By Lemma~\ref{lem:translation}(1) and Remark~\ref{rem:complex}, for every generic $e$ we have
\[
\mathrm{Vol}(\A)=\sum_{\substack{|S|=k\\ e\in\cone(-B_s,\ s\in S)}}|\det B_S| ,
\]
where the sum runs over all the sets $S$ of size $k$, that is, by (U$_C$), over all the vertices of the arrangement of the hyperplanes $\{p_i=0\}$ in the projective space. Thus the three counts of Remark~\ref{rem:complex} can be read on the Gale side. The complex zeros are counted by all the vertices of the arrangement, with weights $|\det B_S|$. The positive zeros of the Viro system are counted by the vertices of the chamber $\overline\Delta$, with weight $1$. The degree is given by the vertices of $\overline\Delta$, with the signs $\omega(S)$. In all three cases, the vertices are selected by the condition $e\in\cone(-B_s,\ s\in S)$. This gives Kushnirenko's theorem in terms of the Gale dual matrices only.
\end{remark}

\begin{remark}[The Gale system]\label{rem:gale-system}
The polyhedron $\Delta$ carries the Gale dual system of Bihan and Sottile \cite{BSGale}. Namely, the map $x\mapsto y$ defined by $p_i(y)=x^{a_i-a_N}$ is a bijection from the positive zeros of $F$ onto the zeros in $\Delta$ of $\psi=(\psi_1,\dots,\psi_k)$, $\psi_j=\sum_iB_{ij}\log p_i$. Indeed, $F(x)=0$ with $x>0$ means that the positive vector $(x^{a_i-a_N})_i$, whose last coordinate is $1$, lies in $\ker C$. Hence this vector is equal to $D(y,1)$ for a unique $y\in\Delta$. Moreover, a positive vector of the form $D(y,1)$ is of the form $(x^{a_i-a_N})_i$ if and only if its logarithm lies in the row space of $A$, that is, if and only if $\psi(y)=0$. Note that $\psi$ does not change if all the $p_i$ are multiplied by the same positive number, since $\sum_iB_i=0$.

The Viro deformation has a simple form on the Gale side. The coefficient matrix of $F^{h,t}$ is $C\,\mathrm{diag}(t^{h_1},\dots,t^{h_N})$, whose kernel is spanned by the columns of $\mathrm{diag}(t^{-h_1},\dots,t^{-h_N})D$. Thus the polyhedron $\Delta$ does not change, and the affine functions $p_i$ are replaced by $t^{-h_i}p_i$. Putting $\nu=-\log t$ and $e=-B^Th$, the Gale system of $F^{h,t}$ is
\[
\psi^{(t)}_j(y)=\sum_iB_{ij}\big(\log p_i(y)-h_i\log t\big)=\psi_j(y)-\nu\,e_j,\qquad\text{that is,}\qquad\psi(y)=\nu\,e .
\]
In other words, the positive zeros of $F^{h,t}$ correspond to the points of $\Delta$ where the Gale map $\psi\colon\Delta\to\R^k$ takes the value $\nu\,e$, and $t\to0$ means $\nu\to+\infty$. Since $\psi$ does not change under a common rescaling of the $p_i$, we may replace them by the linear forms $P_i$ on the compact polytope $Q=\{v\in K:\ \sum_iP_i(v)=1\}$ of Section~\ref{sec:gale-matrices}. This polytope is identified with the closure $\overline\Delta$ of $\Delta$ in the projective space, the point $y\in\Delta$ corresponding to $(y,1)/\sum_ip_i(y)$. Then $\psi$ is bounded on the compact subsets of the interior of $Q$, and it can only be large near the boundary, where some $P_i$ vanish. Near the vertex of $Q$ attached to $S\in V(\Delta)$, the forms $P_s$, $s\in S$, tend to $0$, while the others stay bounded away from $0$. Hence
\[
\psi=\sum_{s\in S}\log P_s\,B_s+O(1),
\]
where $B_s$ is the $s$-th row of $B$. Assume that the $k$ rows $B_s$, $s\in S$, form a basis, and write $e=-\sum_{s\in S}\lambda_sB_s$. Near $S$, the equation $\psi=\nu\,e$ then forces $\log P_s=-\nu\,\lambda_s+O(1)$. For $\nu$ large, this is possible only if all the $\lambda_s$ are positive, that is, if $e$ lies in the interior of $\cone(-B_s,\ s\in S)$. In this case $P_s$ is of order $t^{\lambda_s}$ at the solutions near $S$, and there is exactly one of them, as follows from Theorem~\ref{thm:viro}. By Lemma~\ref{lem:translation}, these vertices $S$ are the complements of the positively decorated simplices $\sigma$ of $T_h$. Moreover, the zero of the Gale map near $S$ corresponds to the zero of $F^{h,t}$ attached to $\sigma$ by Theorem~\ref{thm:viro}, whose Jacobian sign is $s(\sigma)=\epsilon\,\omega(S)$. To summarize, for $t$ small the zeros of the Gale system of $F^{h,t}$ sit near the vertices $S$ of $\overline\Delta$ such that $e\in\cone(-B_s,\ s\in S)$, one near each of them, and Theorem~\ref{thm:B} is the signed count of these zeros. In other words, Theorem~\ref{thm:B} is Theorem~\ref{thm:viro} seen on the Gale side.
\end{remark}

\begin{remark}[Gale form of the vertex reduction]\label{rem:gale-vertex}
Assume (U$_C$) and (H), and let $V$ be the set of indices of the vertices of $P$. By Proposition~\ref{lem:vertex}, $\deg F=\deg F_V$, and Theorem~\ref{thm:Bp} can be applied to $F_V$. The Gale dual matrices of $F_V$ are obtained from $B$ and $D$ by contracting the rows indexed by $V^c=[N]\setminus V$, which are linearly independent. More precisely, a Gale dual matrix of the exponent matrix of $F_V$ has rows the images of the $B_s$, $s\in V$, in $\R^k/\mathrm{span}(B_i,\ i\in V^c)$, and a Gale dual matrix of $C_V$ is given by the restrictions of the forms $P_s$, $s\in V$, to the subspace $\{P_i=0,\ i\in V^c\}$. Therefore the polyhedron of $F_V$ is the face of $\overline\Delta$ cut out by the hyperplanes $\{p_i=0\}$, $i\in V^c$. Only the vertices $S\in V(\Delta)$ of $\overline\Delta$ lying on this face, that is, such that $S\supset V^c$, contribute to the formula for $\deg F$, with the signs $\omega(S)$ up to a global sign. In particular, $\deg F=0$ if this face is empty.
\end{remark}

\subsection{Codimension one or two}\label{sec:codim}
For $k=1$, $\Delta$ is an interval, and $V(\Delta)$ consists of its two endpoints $\{i_0\}$ and $\{i_1\}$. One of them has index $N$ if $\Delta$ is unbounded, as one sees in the model $Q$. The two signs $\sign\det(D_{i_0};D_i)$ and $\sign\det(D_{i_1};D_i)$ are opposite. Hence Theorem~\ref{thm:B} with $e=1$ gives $\deg F=\pm1$ if $B_{i_0}B_{i_1}<0$, and $\deg F=0$ otherwise. This is the parity part of the Descartes rule of signs for circuits \cite{BD}.

We say that a set $S\subset[N]$ is a \emph{coface} if $S\neq\emptyset$, $S\neq[N]$, and the rows $B_s$, $s\in S$, have a positive linear dependence, that is, $\sum_{s\in S}\lambda_sB_s=0$ for some $\lambda\in\R^S_{>0}$.

\begin{lemma}\label{lem:gale-N}
Let $\A$ and $C$ be arbitrary.
\begin{enumerate}
\item A set $G\subset[N]$ is the set of indices of the points of $\A$ on a nonempty proper face of $P$ if and only if $[N]\setminus G$ is a coface.
\item (H) holds if and only if $K_S=\{0\}$ for every coface $S$, that is, when $\Delta\neq\emptyset$, the closure of $\Delta$ in the projective space does not meet the linear subspace $\{p_s=0,\ s\in S\}$.
\item Under (U$_A$), every coface has at least $k+1$ elements; under (U$_C$), $K_S=\{0\}$ whenever $|S|\ge k+1$.
\end{enumerate}
\end{lemma}

\begin{proof}
1. The vectors $\lambda\in\R^N$ such that $\sum_i\lambda_iB_i=0$ form the row space of $A$, that is, they are the vectors $(\ell(a_i))_i$ for the affine functions $\ell$ on $\R^n$. A nonempty proper face is the zero set on $P$ of such a function $\ell$ which is nonnegative on $\A$ and does not vanish identically on $\A$. Its complement in $\A$ is the support of $(\ell(a_i))_i$.

2. By Lemma~\ref{lem:N}(3), (H) means that $\ker C^G$ contains no nonzero nonnegative vector for every proper face $G$. Extended by $0$ on $S=[N]\setminus G$, such a vector is a vector $Dv$ with $v\in K_S\setminus\{0\}$. We conclude by~1.

3. Under (U$_A$), a proper face contains at most $n$ points of $\A$. Under (U$_C$), any $k+1$ rows of $D$ are linearly independent.
\end{proof}

Note that Lemma~\ref{lem:gale-N}(3) is the Gale dual of Lemma~\ref{lem:N}(2), and that it shows again that (U$_A$) and (U$_C$) imply (H). In the same way, (N) says that, for every coface $S$, the restriction of the Gale system to the stratum of $\overline\Delta$ where exactly the $p_s$, $s\in S$, vanish has no zero. This restriction is the Gale system of the truncation $F_{[N]\setminus S}$. We will not use this.

\begin{proposition}\label{prop:k2}
Let $k=2$, assume (U$_C$) and $\Delta\neq\emptyset$, and let $i_1,\dots,i_m$ be the indices of the edges of the polygon $Q$, in counterclockwise order. Then (H) holds if and only if the closed polygon $B_{i_1}B_{i_2}\cdots B_{i_m}$ does not pass through the origin, and in this case $\deg F$ is, up to sign, the winding number of this polygon around the origin.
\end{proposition}

\begin{proof}
Put $i_{m+1}=i_1$. The vertices of $Q$ are the pairs $S_r=\{i_r,i_{r+1}\}$ of consecutive edges.

\emph{The condition (H).} By Lemma~\ref{lem:gale-N}(2), (H) means that $K_S=\{0\}$ for every coface $S$. By Lemma~\ref{lem:gale-N}(3), this only concerns the cofaces with $|S|\le2$, that is, the sets $\{i\}$ with $B_i=0$ and the sets $\{a,b\}$ with $\lambda B_a+\mu B_b=0$ for some $\lambda,\mu>0$. By (U$_C$), at most two of the forms $P_i$ vanish at a nonzero point of $K$, and exactly two of them vanish on an extreme ray. Hence $K_{\{i\}}\neq\{0\}$ if and only if $i$ is an edge of $Q$. Indeed, the nonzero face $K_{\{i\}}$ of $K$ contains an extreme ray, on which $P_i$ is one of the two vanishing forms. Likewise, for $|S|=2$, $K_S\neq\{0\}$ if and only if $S$ is one of the vertices $S_r$. Therefore (H) holds if and only if no row $B_{i_r}$ vanishes and no two consecutive rows $B_{i_r}$, $B_{i_{r+1}}$ have a positive linear dependence, that is, if and only if no segment $[B_{i_r},B_{i_{r+1}}]$ contains the origin.

\emph{The degree.} Assume (H). Let $W$ be a $3\times2$ matrix whose columns form a positive basis of $\ker\lambda$, and let $v_0\in Q$. For $a,b\in[N]$, the matrix $(D_a;D_b;\lambda)(W,v_0)$ has rows $(\nu_a,P_a(v_0))$, $(\nu_b,P_b(v_0))$ and $(0,0,1)$, where $\nu_a=D_aW$ is the gradient of $P_a$ restricted to the plane of $Q$. Hence $\det(D_a;D_b;\lambda)\det(W,v_0)=\det(\nu_a,\nu_b)$. The vectors $\nu_{i_r}$ are inward normals of the consecutive edges of the convex polygon $Q$, so $\det(\nu_{i_r},\nu_{i_{r+1}})>0$ for all $r$. On the other hand, $\det(D_{S_r};\lambda)=\sum_{i\notin S_r}\det(D_{S_r};D_i)$, and all these terms have the same sign. Thus $\sign\det(D_{S_r};D_i)$ is the same for all $r$, and $\omega(S_r)=\pm\sign\det(B_{i_r},B_{i_{r+1}})$ with a sign which does not depend on $r$. (In $\omega(S_r)$ the rows of $B_{S_r}$ and $D_{S_r}$ are taken in increasing order; taking both in the order $(i_r,i_{r+1})$ changes both signs, so it does not change $\omega(S_r)$. Likewise, the orientation of $\ker\lambda$ only affects a global sign.) Let $e$ be as in Theorem~\ref{thm:B}, that is, $-e$ does not lie on any ray $\R_{\ge0}B_i$. If $\det(B_{i_r},B_{i_{r+1}})=0$, then by (H) the two rows point in the same direction. Hence $e\notin\cone(-B_{i_r},-B_{i_{r+1}})$, and the segment $[B_{i_r},B_{i_{r+1}}]$ does not meet the ray $\R_{>0}(-e)$. Otherwise, $e\in\cone(-B_{i_r},-B_{i_{r+1}})$ if and only if this segment crosses the ray $\R_{>0}(-e)$. Therefore, by Theorem~\ref{thm:Bp}, $\deg F$ is, up to sign, the number of segments $[B_{i_r},B_{i_{r+1}}]$ which cross the ray $\R_{>0}(-e)$, counted with the sign of $\det(B_{i_r},B_{i_{r+1}})$. This is the winding number.
\end{proof}

Under (U$_A$), any two rows of $B$ are linearly independent. Hence (H) holds, and the polygon $B_{i_1}\cdots B_{i_m}$ never passes through the origin. Proposition~\ref{prop:k2} is the Gale dual counterpart of Corollary~\ref{cor:n2}. Note however the following difference. On the Newton side, (H) implies that the polygon of the columns $C_i$ indexed by the points of $\A$ on $\partial P$ does not pass through the origin, but it is a stronger condition, since it asks that the columns of each edge lie in an open half-plane. On the Gale side, (H) is exactly the condition that the winding number is defined.

For instance, if the edges of a triangle $Q$ correspond to the rows $(1,0)$, $(0,1)$, $(-1,-1)$, then the triangle with these vertices contains the origin, and $|\deg F|=1$. If the last row is $(1,1)$ instead, then $\deg F=0$.

\begin{example}[Example~\ref{ex:run}, continued: the Gale side]\label{ex:run-gale}
The vectors $(-5,7,-6,4,0)$ and $(-2,2,-1,0,1)$ form a basis of $\ker A$, and the vectors $(7,9,6,0,0)$, $(-11,-9,0,6,0)$, $(2,1,0,0,1)$ form a basis of $\ker C$. We may thus take
\[
B=\begin{pmatrix}-5&-2\\7&2\\-6&-1\\4&0\\0&1\end{pmatrix},\qquad
D=\begin{pmatrix}7&-11&2\\9&-9&1\\6&0&0\\0&6&0\\0&0&1\end{pmatrix}.
\]
Then $p_1=7y_1-11y_2+2$, $p_2=9y_1-9y_2+1$, $p_3=6y_1$, $p_4=6y_2$, $p_5=1$, and $\Delta$ is the unbounded polygon defined by $y_1,y_2>0$, $7y_1-11y_2+2>0$ and $9y_1-9y_2+1>0$. By Proposition~\ref{prop:vertices}, $V(\Delta)$ consists of the three vertices $\{3,4\}$, $\{2,3\}$, $\{1,2\}$, at $(0,0)$, $(0,\frac19)$, $(\frac7{36},\frac{11}{36})$, and of the two vertices at infinity $\{1,5\}$ and $\{4,5\}$, in the directions $(11,7)$ and $(1,0)$. The edges of $Q$ are indexed by $1,\dots,5$ in cyclic order. The signs are given in the following table:
\[
\begin{array}{c|ccccc}
S&\{1,2\}&\{2,3\}&\{3,4\}&\{4,5\}&\{1,5\}\\ \hline
\det B_S&4&5&4&4&-5\\
\sign\det(D_S;D_{i_S})&+1&+1&+1&+1&-1
\end{array}
\]
Thus $\omega(S)=+1$ for all $S\in V(\Delta)$. For $e=(0,1)$, the vertices such that $e\in\cone(-B_s,\ s\in S)$ are $\{1,2\}$ and $\{3,4\}$, and for $e=(1,0)$ they are $\{1,5\}$ and $\{2,3\}$. Hence $\deg(\Delta)=2$, and $\deg F=\deg(\Delta)$ for these choices of $B$ and $D$. For the heights $h=(-1,0,0,0,0)$ of Example~\ref{ex:run}, the vector $e=-B^Th=B_1=(-5,-2)$ selects $\{4,5\}$ and $\{2,3\}$. As predicted by Lemma~\ref{lem:translation}, these are the complements of the positively decorated triangles $\{1,2,3\}$ and $\{1,4,5\}$ of $T_h$. Finally, the closed polygon $B_1B_2B_3B_4B_5$ turns twice around the origin, as predicted by Proposition~\ref{prop:k2}.
\end{example}

\begin{example}[El Hilany's system]\label{ex:EH}
El Hilany \cite{EH} constructed a system of two equations in two variables with support
\[
a_1=(0,0),\quad a_2=(6,0),\quad a_3=(3,6),\quad a_4=(-14,7),\quad a_5=(-12,9),
\]
and seven positive solutions. Here $n=k=2$ and $N=5$. A Gale dual matrix of $A$ has rows
\[
B_1=\big(-\tfrac{11}4,-\tfrac94\big),\quad B_2=\big(\tfrac{35}{12},\tfrac{11}4\big),\quad B_3=\big(-\tfrac76,-\tfrac32\big),\quad B_4=(1,0),\quad B_5=(0,1).
\]
For the coefficients of \cite{EH}, the positive chamber is the pentagon of the arrangement of the five lines $\{p_i=0\}$, with edges in the cyclic order $2,1,3,4,5$. The arguments of $B_2,B_1,B_3,B_4,B_5$ are approximately $43^\circ$, $219^\circ$, $232^\circ$, $0^\circ$ and $90^\circ$. The successive turning angles of the closed polygon $B_2B_1B_3B_4B_5$ are approximately $176^\circ$, $13^\circ$, $128^\circ$, $90^\circ$ and $-47^\circ$, which add up to $360^\circ$. Hence the polygon turns once around the origin, and $\deg F=\pm1$ by Proposition~\ref{prop:k2}. In other words, the number of nondegenerate positive solutions is odd. On the other hand, the upper bound of \cite{PartI} for this system, computed with the best choice of the flag of relations, is $8$. By Corollary~\ref{cor:parity}, the system has at most $7$ nondegenerate positive solutions, and El Hilany's system shows that this refined bound is attained. Note also that $|\deg F|=1$ is far from the largest possible degree $2$ for $n=k=2$ (Theorem~\ref{thm:E}). Thus the parity statement is useful even when the degree is small.
\end{example}

\section{A Descartes rule of signs for the degree}\label{sec:descartes}

Let $(c_1,\dots,c_r)$ be a cyclic sequence of nonzero real numbers. Its number of \emph{cyclic sign changes} $V_{\rm cyc}(c_1,\dots,c_r)$ is the number of indices $j\in\{1,\dots,r\}$ such that $c_jc_{j+1}<0$, where $c_{r+1}=c_1$. In other words, it is the number of sign changes of the sequence $(c_1,\dots,c_r,c_1)$. Note that it is always even.

In higher dimension we use the following notion. We say that $m$ vectors of $\R^{m-1}$ form a \emph{sign change} if they positively span $\R^{m-1}$. For $m=2$, this means that two real numbers are nonzero and have opposite signs, so this is an ordinary sign change.

On the Newton side, let $g_j=\sum_lg_{lj}f_l$, $j=1,\dots,n-1$, be $n-1$ linear combinations of the components of $F$, and for $i\in[N]$ let $\gamma_i\in\R^{n-1}$ be the vector of the coefficients of $x^{a_i}$ in $g_1,\dots,g_{n-1}$. Let $T$ be a triangulation of $\partial P$ as in Theorem~\ref{thm:Ap}. We say that an $(n-1)$-simplex $\tau$ of $T$ is a \emph{sign change} of $(g_1,\dots,g_{n-1})$ if the $n$ vectors $\gamma_s$, $s\in\tau$, form a sign change. The number of sign changes of $(g_1,\dots,g_{n-1})$ on $T$ is the number of these simplices,
\[
V_T(g_1,\dots,g_{n-1})=\#\big\{\tau\in T:\ \dim\tau=n-1,\ \tau\ \text{is a sign change of}\ (g_1,\dots,g_{n-1})\big\},
\]
that is, the sum over all the $(n-1)$-simplices $\tau$ of $T$ of $1$ if $\tau$ is a sign change and $0$ otherwise. Equivalently, $\tau$ is a sign change if and only if the truncated system $g_1|_\tau=\dots=g_{n-1}|_\tau=0$ has a positive solution, where $g_j|_\tau$ is the sum of the terms of $g_j$ with exponents in $\tau$ (see the proof of Theorem~\ref{thm:descartes}). Under (U$_A$), we can take for $T$ the boundary complex of $P$, whose $(n-1)$-simplices are the facets of $P$. By the vertex reduction, we can also take for $T$ a triangulation of $\partial P$ whose vertices are the vertices of $P$. We say that $(g_1,\dots,g_{n-1})$ is \emph{admissible} for $T$ if there is no vector $\lambda\ge0$, supported on at most $n-1$ vertices of $T$, such that
\[
\sum_i\lambda_i\gamma_i=0\qquad\text{and}\qquad\sum_i\lambda_iC_i\neq0 .
\]
This holds in particular if any $n-1$ of the vectors $\gamma_s$, $s$ a vertex of $T$, are linearly independent, which is the case for generic $g_j$ under (U$_C$). In general, admissibility holds for generic $g_j$ (see the proof of Theorem~\ref{thm:descartes}). For $n=2$, a linear combination $g$ is admissible if and only if its coefficient $g_v$ is nonzero at every vertex $v$ of $T$, since $C_v\neq0$ by (H).

On the Gale side, let $b_1,\dots,b_{k-1}\in\ker A$ be affine relations of $\A$, and for $i\in[N]$ let $\beta_i=(b_1(i),\dots,b_{k-1}(i))\in\R^{k-1}$. We say that a vertex $S\in V(\Delta)$ is a \emph{sign change} of $(b_1,\dots,b_{k-1})$ if the $k$ vectors $\beta_s$, $s\in S$, form a sign change, and we denote by
\[
V_\Delta(b_1,\dots,b_{k-1})=\#\big\{S\in V(\Delta):\ S\ \text{is a sign change of}\ (b_1,\dots,b_{k-1})\big\}
\]
the number of these vertices. We say that $(b_1,\dots,b_{k-1})$ is \emph{admissible} if there is no vector $\lambda\ge0$, supported on at most $k-1$ indices, such that $\sum_i\lambda_i\beta_i=0$ and $\sum_i\lambda_ib(i)\neq0$ for some $b\in\ker A$. This holds in particular if any $k-1$ of the vectors $\beta_i$ are linearly independent, which is the case for generic $b_j$ under (U$_A$), and admissibility holds for generic $b_j$ in general. For $k=2$, an affine relation $b$ is admissible if and only if $b(i)\neq0$ for every $i$ such that the linear form $b'\mapsto b'(i)$ on $\ker A$ is nonzero.

\begin{theorem}\label{thm:descartes}
Assume (H). Statements~1 and~3 hold, and if moreover (U$_C$) holds, statements~2 and~4 hold.
\begin{enumerate}
\item \emph{Newton side.} For every triangulation $T$ of $\partial P$ as in Theorem~\ref{thm:Ap}, and for all linear combinations $g_1,\dots,g_{n-1}$ of the components of $F$ which are admissible for $T$,
\[
|\deg F|\le\tfrac12V_T(g_1,\dots,g_{n-1}).
\]
\item \emph{Gale side.} For all admissible affine relations $b_1,\dots,b_{k-1}$ of $\A$,
\[
|\deg F|\le\tfrac12V_\Delta(b_1,\dots,b_{k-1})\quad\text{if }\Delta\neq\emptyset,\qquad\deg F=0\quad\text{if }\Delta=\emptyset.
\]
\item \emph{Two variables.} If $n=2$ and $v_1,\dots,v_r$ are the vertices of the Newton polygon in cyclic order, then for every linear combination $g=\lambda_1f_1+\lambda_2f_2$, $\lambda_1,\lambda_2\in\R$, whose coefficients $g_i=\langle(\lambda_1,\lambda_2),C_i\rangle$ at the vertices of the Newton polygon are nonzero,
\[
|\deg F|\le\tfrac12V_{\rm cyc}(g_{v_1},\dots,g_{v_r}).
\]
\item \emph{Codimension two.} If $k=2$ and $i_1,\dots,i_m$ are the edges of $\overline\Delta$ in cyclic order, then for every admissible affine relation $b$ of $\A$,
\[
|\deg F|\le\tfrac12V_{\rm cyc}(b_{i_1},\dots,b_{i_m}).
\]
\end{enumerate}
\end{theorem}

Statements~3 and~4 are the special cases $n=2$ and $k=2$ of statements~1 and~2. Indeed, for $n=2$ we take for $T$ the triangulation of $\partial P$ by the edges $\{v_l,v_{l+1}\}$ of the Newton polygon. An edge is a sign change of $g$ if and only if $g_{v_l}g_{v_{l+1}}<0$, so that $V_T(g)=V_{\rm cyc}(g_{v_1},\dots,g_{v_r})$.  Similarly, for $k=2$ the vertices of $\overline\Delta$ are the pairs of consecutive edges $\{i_l,i_{l+1}\}$, and $V_\Delta(b)=V_{\rm cyc}(b_{i_1},\dots,b_{i_m})$. For $n=1$, Theorem~\ref{thm:A} gives the parity part of the classical Descartes rule: $|\deg F|=1$ if the extreme coefficients have opposite signs, and $\deg F=0$ otherwise (Section~\ref{sec:onetwo}). In this sense statements~1 and~2 are Descartes rules of signs for the degree in any dimension. The number of sign changes of a sequence of coefficients is replaced by the number of simplices of a triangulation of the boundary of the Newton polytope (the facets, in the generic case), or of vertices of $\overline\Delta$, at which the coefficient vectors form a sign change, and the bound only depends on the signs of the maximal minors of $A$ and of the coefficient matrix of $g_1,\dots,g_{n-1}$, or of $B$ and $D$.

\begin{proof}
Statements~1 and~2 are the bounds given by Propositions~\ref{prop:ends-newton} and~\ref{prop:ends-gale}, which express $\deg F$ as half a signed count. Statements~3 and~4 follow, as explained above.

1. Let $M\in GL_n(\R)$ be a matrix whose first $n-1$ rows are the vectors $(g_{1j},\dots,g_{nj})$, $j=1,\dots,n-1$. Then the first $n-1$ components of $MF$ are $g_1,\dots,g_{n-1}$, $\deg(MF)=\pm\deg F$, and $MF$ satisfies (H). The matrix $C'$ formed by the first $n-1$ rows of $MC$ has columns $\gamma_i$. A combination $\sum_i\lambda_iMC_i$ is a multiple of $e_n$ if and only if $\sum_i\lambda_iC_i$ lies in the line $L$ of the vectors $w\in\R^n$ such that $\sum_lg_{lj}w_l=0$ for $j=1,\dots,n-1$. Admissibility says exactly that $L$ meets the cones spanned by $n-1$ columns $C_s$, $s$ a vertex of $T$, only at the origin. Hence neither $e_n$ nor $-e_n$ lies in such a cone spanned by columns of $MC$, which is what is used in the proof of Proposition~\ref{prop:ends-newton}, since only the simplices of $T$ are involved in Lemma~\ref{lem:PL}. Note that for generic $g_j$, $L$ is a generic line, which meets these cones only at the origin, since each of them lies in a proper subspace of $\R^n$; thus generic $g_j$ are admissible. Therefore Proposition~\ref{prop:ends-newton} applied to $MF$ expresses $2\deg F$ as $\pm$ a sum of $V_T(g_1,\dots,g_{n-1})$ signs, which gives the bound. Finally, let $\tau$ be an $(n-1)$-simplex of $T$. The truncated system $g_1|_\tau=\dots=g_{n-1}|_\tau=0$ has $n$ monomials with affinely independent exponents and coefficient vectors $\gamma_s$, $s\in\tau$. Its vector of monomials must be, up to a positive factor, the positive vector spanning the kernel of its coefficient matrix. Hence it has a positive solution if and only if the vectors $\gamma_s$, $s\in\tau$, positively span $\R^{n-1}$, that is, if and only if $\tau$ is a sign change.

2. If $\Delta=\emptyset$, then $\deg F=0$ by Theorem~\ref{thm:Bp}. Otherwise, we complete $b_1,\dots,b_{k-1}$ into a basis $b_1,\dots,b_k$ of $\ker A$, and take for $B$ the matrix with these columns. Its first $k-1$ columns form the matrix $B'$ with rows $\beta_i$. The rows of $B$ are the coordinates, in the dual basis, of the linear forms $b\mapsto b(i)$ on $\ker A$, and a combination of them is a multiple of $e_k$ if and only if the corresponding linear form vanishes on the hyperplane $W=\mathrm{span}(b_1,\dots,b_{k-1})$ of $\ker A$. Admissibility says exactly that the line of the linear forms vanishing on $W$ meets the cones spanned by $k-1$ of the forms $-B_i$ only at the origin. Hence neither $e_k$ nor $-e_k$ lies in a cone spanned by $k-1$ of the vectors $-B_i$. For generic $b_j$, $W$ is a generic hyperplane and this line is a generic line, so generic $b_j$ are admissible. Then Proposition~\ref{prop:ends-gale} expresses $2\deg F$ as $\pm$ a sum of $V_\Delta(b_1,\dots,b_{k-1})$ signs, which gives the bound.
\end{proof}

In particular, if some linear combination of two equations in two variables has coefficients of the same sign at all the vertices of the Newton polygon, then $\deg F=0$. Observe that, in statement~3, only the vertices of the Newton polygon matter, that the sign changes are counted cyclically, and that one may first take any linear combination of the equations.

\begin{example}[Example~\ref{ex:run}, continued: the Descartes rule]
For the running example, the coefficients of $f_1$ at the vertices $a_1,\dots,a_5$ are $3,-3,1,1,-3$. This sequence has $4$ cyclic sign changes, so statement~3 gives $|\deg F|\le2$, which is an equality. On the Gale side, the sum of the two columns of $B$ in Example~\ref{ex:run-gale} is the affine relation $b=(-7,9,-7,4,1)$ of $\A$, and the edges of $\Delta$ are indexed by $1,\dots,5$ in cyclic order. The sequence $b$ has $4$ cyclic sign changes, and statement~4 also gives $|\deg F|\le2$.
\end{example}

\section{An upper bound for the degree}\label{sec:bound}

For $d\ge1$ and $N\ge d+1$, let $M(d,N)$ be as in Theorem~\ref{thm:C}. For example, we have $M(1,N)=2$, $M(2,N)=N$, $M(3,N)=2N-4$ and $M(4,N)=N(N-3)/2$. Recall that, by the Upper Bound Theorem \cite{McM}, $M(d,N)$ is the largest number of facets of a $d$-polytope with $N$ vertices and, dually, the largest number of vertices of a $d$-polytope with $N$ facets. To prove Theorem~\ref{thm:C}, we bound the number of terms of the half signed counts of Propositions~\ref{prop:ends-newton} and~\ref{prop:ends-gale}, for a well chosen change of basis.

\begin{proof}[Proof of Theorem~\ref{thm:C}]
If $n=1$, then $|\deg F|\le1$ by Section~\ref{sec:onetwo}, and if $k=1$, then $|\deg F|\le1$ by Section~\ref{sec:codim}. Assume $n\ge2$. Let $M\in GL_n(\R)$ be generic, so that Proposition~\ref{prop:ends-newton} applies to $MF$, with the boundary complex of $P$ as triangulation. By the proof of this proposition, a facet $\Gamma$ contributes a term if and only if $e_n$ or $-e_n$ lies in $\cone(MC_s,\ s\in\Gamma)$, that is, if and only if the line $L=M^{-1}(\R e_n)$ meets $\cone(C_s,\ s\in\Gamma)$ outside the origin. Hence $2|\deg F|$ is at most the number of facets $\Gamma$ of $P$ such that $L$ meets the simplicial cone $\cone(C_s,\ s\in\Gamma)$. Since $n\ge2$, the union of a simplicial cone of dimension $n$ and of its opposite is not $\R^n$. Therefore we may choose $M$ such that $L$ does not meet the cone of a given facet. Then $2|\deg F|$ is less than the number of facets of $P$, which is at most $M(n,N)$, since $P$ is an $n$-polytope with at most $N$ vertices. Thus $|\deg F|<\frac12M(n,N)$.

In the same way, assume $k\ge2$ and $\Delta\neq\emptyset$ (if $\Delta=\emptyset$, then $\deg F=0$). By Proposition~\ref{prop:ends-gale}, applied after a generic change of basis of $\ker A$, $2|\deg F|$ is at most the number of $S\in V(\Delta)$ such that a generic line meets the simplicial cone $\cone(-B_s,\ s\in S)$ outside the origin. Moreover, the line may be chosen so that it does not meet the cone of a given $S$. Hence $2|\deg F|<\#V(\Delta)\le M(k,N)$, because $V(\Delta)$ is the vertex set of the simple $k$-polytope $Q$, which has at most $N$ facets.

Finally, if $k=2$ or $n=2$, then $\min(M(n,N),M(k,N))=N$, and the bound reads $|\deg F|\le\lceil N/2\rceil-1=\lfloor(N-1)/2\rfloor$.
\end{proof}

\begin{proof}[Proof of Theorem~\ref{thm:E}]
By Theorem~\ref{thm:C}, $\mathcal D(n,1)$ and $\mathcal D(1,n)$ are contained in $\{0,1\}$, and $\mathcal D(n,2)$ and $\mathcal D(2,n)$ are contained in $\{0,\dots,w\}$, where $w=\lfloor(N-1)/2\rfloor=\lfloor n/2\rfloor+1$ with $N=n+3$. The value $0$ is attained by any map with $\Delta=\emptyset$, for instance with all the columns of $C$ in an open half-space. For $n=1$, the value $1$ is attained by any polynomial whose extreme coefficients have opposite signs. For $k=1$, it is attained since $\mathcal D(n,1)=\mathcal D(1,n)$ by Theorem~\ref{thm:D}, or directly by a circuit with $B_{i_0}B_{i_1}<0$ (Section~\ref{sec:codim}).

\emph{The case $n=2$.} Let $\A$ be the set of vertices of a convex polygon with $N$ vertices, numbered counterclockwise. It satisfies (U$_A$). Let $1\le j\le w$, and let $C_r$ be a small generic perturbation of the vector $(\cos\frac{2\pi jr}N,\sin\frac{2\pi jr}N)$, $r=1,\dots,N$, so that (U$_C$) holds. Since $2\pi j/N<\pi$, the angle from $C_r$ to $C_{r+1}$ (indices modulo $N$) is close to $2\pi j/N$, hence lies in $(0,\pi)$, and the sum of these angles is $2\pi j$. Thus the closed polygon $C_1C_2\cdots C_N$ winds $j$ times around the origin, and $\deg F=j$ by Corollary~\ref{cor:n2}.

\emph{The case $k=2$.} Let $\theta_1<\dots<\theta_N$ be angles in $[0,2\pi)$, and let $D^0$ be the $N\times3$ matrix with rows $(-\cos\theta_i,-\sin\theta_i,1)$. The inequalities $\langle D^0_i,(y,1)\rangle>0$ define the polygon with $N$ sides circumscribed about the unit circle. No three of the tangent lines are concurrent, so any three rows of $D^0$ are linearly independent. Let $D=D^0G$ for a matrix $G\in GL_3(\R)$ such that $D_N=(0,0,1)$. It has the same cone $K$ up to a linear isomorphism, and $Q$ is a polygon with $N$ sides, in the cyclic order $1,\dots,N$. Let $C$ be a matrix whose rows form a basis of the orthogonal complement of the columns of $D$. Then $D$ is a Gale dual matrix of $C$, and $C$ satisfies (U$_C$). Next, for $1\le j\le w$, let $B$ be the $N\times2$ matrix whose rows are small generic perturbations of $(\cos\frac{2\pi jr}N,\sin\frac{2\pi jr}N)$, $r=1,\dots,N$, corrected so that its rows sum to zero. This is possible since the unperturbed rows sum to zero. Let $A$ be a matrix whose rows form a basis of the orthogonal complement of the columns of $B$, the first row being $(1,\dots,1)$. Its columns are of the form $(1,a_i)$ with $a_i\in\R^n$, and $B$ is a Gale dual matrix of $A$. Moreover, (U$_A$) holds since any two rows of $B$ are linearly independent. As in the case $n=2$, the closed polygon $B_1\cdots B_N$ winds $j$ times around the origin, and $|\deg F|=j$ by Proposition~\ref{prop:k2}. The exponents are real. However, since the degree only depends on the oriented matroids (Theorem~\ref{cor:om}), they may be perturbed to rational numbers and then multiplied by a common integer. Hence integer exponents give the same set of degrees. Note that the case $k=2$ also follows from the case $n=2$ and Theorem~\ref{thm:D}.
\end{proof}

For $n,k\ge3$ we do not know whether the bound of Theorem~\ref{thm:C} is sharp. It gives $|\deg F|\le4$ for $n=k=3$ and $|\deg F|\le13$ for $n=k=4$.
\section{H-duality and symmetry}\label{sec:h}

Assume (U$_A$), (U$_C$) and $\Delta\neq\emptyset$. Let $v_0$ be an interior point of the cone $K$, so that $\langle D_i,v_0\rangle>0$ for all $i$. We put $\Lambda=\mathrm{diag}(\langle D_i,v_0\rangle)$ and $\tilde D_i=D_i/\langle D_i,v_0\rangle$. The points $\tilde D_i$ lie on the affine hyperplane $\{\langle\cdot,v_0\rangle=1\}$ of $\R^{k+1}$. Let $\delta_i\in\R^k$ be their coordinates in an affine coordinate system of this hyperplane. Following \cite[Section~5]{BDF24}, we call
\begin{equation}\label{eq:h-dual}
H=(H_1,\dots,H_k)\colon\R^k_{>0}\to\R^k,\qquad H_j(z)=\sum_{i=1}^NB_{ij}\,z^{\delta_i},
\end{equation}
an \emph{H-dual} of $F$. This is a fewnomial map with $k$ components in $k$ variables, with $N=k+n+1$ monomials, exponent set $\{\delta_1,\dots,\delta_N\}$ and coefficient matrix $B^T$. In \cite{BDF24}, the point $v_0$ is attached to a positive zero $x_0$ of $F$ by $Dv_0=(x_0^{a_i})_i$. In other words, up to a positive factor, $v_0=(y_0,1)$, where $y_0\in\Delta$ is the zero of the Gale system $\psi$ corresponding to $x_0$ (Remark~\ref{rem:gale-system}). Moreover, in \cite{BDF24} the zero is $x_0=(1,\dots,1)$. This can always be achieved by the change of variables $x\mapsto x_0\star x$, $(x_0\star x)_j=x_{0,j}x_j$, and it means that the columns of $C$ sum to zero; then $Dv_0=(1,\dots,1)$. Since $\sum_iB_i=0$, every H-dual vanishes at $(1,\dots,1)$, and with this choice of $v_0$ this zero corresponds to $x_0$. It is used in \cite{BDF24} to compare the local multiplicities of $F$ at $x_0$ and of $H$ at $(1,\dots,1)$.

\begin{lemma}\label{lem:h-dual}
\begin{enumerate}
\item $\Lambda C^T$ is a Gale dual matrix of the exponent matrix of $H$, and $A^TG$ is a Gale dual matrix of its coefficient matrix $B^T$, for some $G\in GL_{n+1}(\R)$.
\item $H$ satisfies (U$_A$) and (U$_C$), and its positive chamber $\Delta_H$ is nonempty.
\item The vertices of $\overline{\Delta_H}$ are the sets $\Gamma\subset[N]$ which are facets of $\conv(\A)$.
\item The facets of the Newton polytope $\conv(\delta_1,\dots,\delta_N)$ of $H$ are the sets $S\in V(\Delta)$.
\end{enumerate}
\end{lemma}

\begin{proof}
1. The row space of the matrix with columns $(1,\delta_i)$ consists of the vectors $(\langle\tilde D_i,v\rangle)_i$, $v\in\R^{k+1}$. In other words, it is $\Lambda^{-1}\ker C$. Its orthogonal complement $\Lambda(\ker C)^\perp$ is spanned by the columns of $\Lambda C^T$. On the other hand, the kernel of $B^T$ is $(\ker A)^\perp$, that is, the row space of $A$, which is spanned by the columns of $A^T$.

2. By~1, (U$_A$) for $H$ means that any $n$ rows of $\Lambda C^T$ are linearly independent, which is (U$_C$) for $F$. Likewise, (U$_C$) for $H$ is (U$_A$) for $F$. As in Section~\ref{sec:gale}, $\Delta_H\neq\emptyset$ if and only if the kernel of the coefficient matrix $B^T$ of $H$ contains a positive vector. This kernel is the row space of $A$, which contains $(1,\dots,1)$.

3. By Proposition~\ref{prop:vertices}, a set $\Gamma$ with $|\Gamma|=n$ is a vertex of $\overline{\Delta_H}$ if and only if the signs of the determinants of the rows of $A^TG$ indexed by $\Gamma$, followed by the row $i$, do not depend on $i\notin\Gamma$. These signs are $\sign\det G\cdot\sign\det A^{\Gamma\cup\{i\}}$. They do not depend on $i$ exactly when all the other points of $\A$ lie on the same side of the affine hull of $\Gamma$, that is, when $\Gamma$ is a facet.

4. Recall that the $\delta_i$ are the coordinates of the points $\tilde D_i$ in an affine coordinate system of the hyperplane $\{\langle\cdot,v_0\rangle=1\}$. Homogenizing this coordinate system, we get an invertible $(k+1)\times(k+1)$ matrix $M'$ such that $(1,\delta_i)=M'\tilde D_i=M'D_i/\langle D_i,v_0\rangle$ for all $i$. Hence, for $|S|=k$ and $i\notin S$,
\[
\det\big((1,\delta_s)_{s\in S},(1,\delta_i)\big)=\det M'\cdot\frac{\det(D_S;D_i)}{\prod_{s\in S}\langle D_s,v_0\rangle\cdot\langle D_i,v_0\rangle},
\]
with positive denominators. A set $S$ of size $k$ is a facet of $\conv(\delta_1,\dots,\delta_N)$ if and only if the sign of the left-hand side does not depend on $i\notin S$. This means that the signs of $\det(D_S;D_i)$, $i\notin S$, are equal, that is, by Proposition~\ref{prop:vertices}, that $S\in V(\Delta)$.
\end{proof}

Items~3 and~4 express the involutivity of the H-duality. In the same way, one checks that an H-dual of $H$ is the map $F$, with its exponents replaced by their images under a projective transformation which is defined on $\conv(\A)$, and with its coefficient columns $C_i$ multiplied by the positive numbers $\langle D_i,v_0\rangle$. Neither change affects the signs of the maximal minors, so this H-dual has the same degree as $F$.

\begin{theorem}\label{thm:h-duality}
Assume (U$_A$), (U$_C$) and $\Delta\neq\emptyset$. Then $\deg(\Delta_H)$ is, up to sign, the sum of Theorem~\ref{thm:A} for $F$, and consequently
\[
\deg H=\pm\deg F .
\]
\end{theorem}

\begin{proof}
We apply the formula of Theorem~\ref{thm:B} to $H$, with the Gale dual matrices of Lemma~\ref{lem:h-dual}(1) and a generic vector $e_H\in\R^n$. By Lemma~\ref{lem:h-dual}(3), the sum runs over the facets $\Gamma$ of $\conv(\A)$. Since $\Lambda$ is positive, the rows of $\Lambda C^T$ indexed by $\Gamma$ have a determinant of sign $\sign\det C^\Gamma$. The second sign in $\omega(\Gamma)$ is $\sign\det G\cdot\varepsilon_\Gamma$. Finally, $e_H\in\cone(-\Lambda_{ss}C_s,\ s\in\Gamma)$ if and only if $-e_H\in\cone(C_s,\ s\in\Gamma)$. Therefore $\deg(\Delta_H)$ is $\sign\det G$ times the sum of Theorem~\ref{thm:A} for $F$, computed with $u=-e_H$. By Theorems~\ref{thm:A} and~\ref{thm:B}, we conclude that $\deg H=\pm\deg(\Delta_H)=\pm\deg F$.
\end{proof}

Thus the H-duality exchanges the two formulas: the formula of Theorem~\ref{thm:B} for $H$, over the vertices of $\overline{\Delta_H}$, is the formula of Theorem~\ref{thm:A} for $F$, over the facets of $\conv(\A)$. Note that only the signed counts are preserved. For instance, consider $x^3-2x+1=(x-1)(x^2+x-1)$, with the exponents ordered as $a_1=0$, $a_2=1$, $a_3=3$, so that $B$ has rows $2,-3,1$. It has two positive roots of opposite signs. With $v_0=(1,1)$ and $D$ with rows $(2,1),(1,1),(0,1)$, its H-dual is $2z^{2/3}-3z^{1/2}+1$, which, with $t=z^{1/6}$, is $(t-1)(2t^3-t^2-t-1)$. It also has two positive roots of opposite signs, and both degrees vanish. The same translation also exchanges the Viro formula and the degree at infinity. Indeed, multiply the coefficient $B_i$ of $z^{\delta_i}$ in~\eqref{eq:h-dual} by $t^{g_i}$, where $g\in\R^N$ is such that $u=C\Lambda g$ is generic as in Theorem~\ref{thm:A}. By Lemma~\ref{lem:translation} applied to $H$, with the Gale dual matrices of Lemma~\ref{lem:h-dual}(1), the simplices of the regular triangulation of $\{\delta_1,\dots,\delta_N\}$ induced by $g$ which are positively decorated by $B^T$ are the complements of the facets $\Gamma$ of $\conv(\A)$ such that $u\in\cone(C_s,\ s\in\Gamma)$. Hence the sum of Theorem~\ref{thm:A} for $F$ is the Viro formula for $H$, which gives another explanation of Section~\ref{sec:coning}.

\begin{example}[Example~\ref{ex:run}, continued: the H-dual]
We take the matrices $B$ and $D$ of Example~\ref{ex:run-gale}, and $v_0=(\frac16,\frac16,1)$, which lies in the interior of $K$ since $Dv_0=(\frac43,1,1,1,1)$. As affine coordinates on the hyperplane $\{\langle\cdot,v_0\rangle=1\}$, we take $\frac43$ times the first two coordinates. Then $\delta_1=(7,-11)$, $\delta_2=(12,-12)$, $\delta_3=(8,0)$, $\delta_4=(0,8)$, $\delta_5=(0,0)$, and
\[
H(z)=\big(-5z_1^7z_2^{-11}+7z_1^{12}z_2^{-12}-6z_1^8+4z_2^8,\ -2z_1^7z_2^{-11}+2z_1^{12}z_2^{-12}-z_1^8+1\big).
\]
Its Newton polygon is the pentagon with vertices $\delta_5,\delta_1,\delta_2,\delta_3,\delta_4$ in counterclockwise order, and its edges are the five vertices of $\overline\Delta$ (Lemma~\ref{lem:h-dual}(4)). Its coefficient columns are the rows $B_i$, and the closed polygon $B_5B_1B_2B_3B_4$ turns twice around the origin. Hence $\deg H=2=\deg F$ by Corollary~\ref{cor:n2}. Observe that the computation of $\deg H$ by Theorem~\ref{thm:A} is exactly the computation of $\deg F$ by Theorem~\ref{thm:B} in Example~\ref{ex:run-gale}.
\end{example}

\begin{proof}[Proof of Theorem~\ref{thm:D}]
A map with $\Delta=\emptyset$ has degree $0$, and such maps exist for all $(n,k)$, so $0$ belongs to both sets. Let $F$ be of type $(n,k)$, satisfying (U$_A$) and (U$_C$), with $\Delta\neq\emptyset$. Its H-dual $H$ has type $(k,n)$, it satisfies (U$_A$) and (U$_C$) by Lemma~\ref{lem:h-dual}(2), and $|\deg H|=|\deg F|$ by Theorem~\ref{thm:h-duality}. Thus $\mathcal D(n,k)\subset\mathcal D(k,n)$, and exchanging $n$ and $k$ gives the equality. The exponents of $H$ are real. However, since the degree only depends on signs of minors (Theorem~\ref{cor:om}), they may be perturbed to rational numbers and then, by the change of variables $z\mapsto z^q$, to integers. Hence the statement also holds for maps with integer exponents.
\end{proof}

\begin{remark}
In terms of oriented matroids, the H-duality is the map $(M_A,M_C)\mapsto(M_C^*,M_A^*)$, where $M_A$ and $M_C$ are the oriented matroids of the columns of $A$ and $C$, and $M^*$ denotes the dual, or orthogonal, oriented matroid of $M$ \cite[Section~3.4]{BLSWZ}. Recall that if $M$ is realized by the columns of a matrix, then $M^*$ is realized by the rows of a Gale dual matrix, and the vectors of $M^*$ are the covectors of $M$. Thus $M_A^*$ and $M_C^*$ are the oriented matroids of the rows of $B$ and of $D$. Recall also that an oriented matroid is \emph{acyclic} if it has no nonzero vector with nonnegative signs. For an oriented matroid realized by vectors $v_1,\dots,v_N$, this means that there is no linear dependence $\sum_i\lambda_iv_i=0$ with $\lambda\ge0$, $\lambda\neq0$, that is, by Gordan's theorem, that the $v_i$ lie in an open half-space. The oriented matroid $M_A$ is acyclic, since the first coordinate of the columns $(1,a_i)$ is positive. On the other hand, $M_C^*$ is acyclic if and only if $\ker C$ contains a positive vector, that is, if and only if $\Delta\neq\emptyset$. Therefore this map is an involution on the pairs with $\Delta\neq\emptyset$, which preserves $|\deg F|$. For $n\le2$ or $k\le2$, the largest degree is known independently on both sides. It is $1$ for $n=1$ or $k=1$, and $\lfloor(N-1)/2\rfloor$ for $n=2$ or $k=2$, by the winding number of the polygon of the columns of $C$ in one case and of the rows of $B$ in the other. We do not know whether the largest number of positive zeros, rather than the largest degree, is symmetric in $n$ and $k$.
\end{remark}

\bibliographystyle{amsplain}
\bibliography{article-degree}

\end{document}